\documentclass{imanum}
\usepackage{graphicx}
\usepackage{subfigure}
\usepackage{algorithm}
\usepackage{algorithmicx}
\usepackage{algpseudocode}

\jno{drnxxx}

\begin{document}

\title{Second order asymptotical regularization methods for inverse problems in partial differential equations}
\shorttitle{Second order asymptotical regularization}

\author{%
{\sc
Ye Zhang\thanks{Email: ye.zhang@mathematik.tu-chemnitz.de} \\[2pt]
School of Science and Technology, \"{O}rebro University,\\
70182 \"{O}rebro, Sweden and\\
Faculty of Mathematics, Chemnitz University of Technology,\\
09107 Chemnitz, Germany \\[6pt]
{\sc and}\\[6pt]
{\sc Rongfang Gong}\thanks{Corresponding author. Email: grf\_math@nuaa.edu.cn}\\[2pt]
Department of Mathematics, Nanjing University of Aeronautics and Astronautics,
211106 Nanjing, China}
}
\shortauthorlist{Y. Zhang and R. Gong}

\maketitle

\begin{abstract}
{We develop Second Order Asymptotical Regularization (SOAR) methods for solving
inverse source problems in elliptic partial differential equations with both Dirichlet
and Neumann boundary data. We show the convergence results of SOAR with the fixed damping
parameter, as well as with a dynamic damping parameter, which is a continuous analog
of Nesterov's acceleration method. Moreover, by using Morozov's discrepancy principle
together with a newly developed total energy discrepancy principle, we prove that the
approximate solution of SOAR weakly converges to an exact source function as the measurement noise goes to zero. A damped symplectic scheme, combined with the finite element method, is developed for the numerical implementation of SOAR, which yields a novel iterative regularization scheme for solving inverse source problems. Several numerical examples are given to show the accuracy and the acceleration effect of SOAR. A comparison with the state-of-the-art methods is also provided. }
{Inverse source problems; Partial differential equations; Asymptotical regularization; Convergence; Finite element methods; Symplectic methods.}
\end{abstract}

\section{Introduction}
\label{sec;introduction}

In this paper, inspired by the asymptotical regularization (\cite{Vainikko1986,Tautenhahn-1994,ZhangHof2018}), we establish a new framework for stably solving inverse problems in partial differential equations (PDEs). To present the ideas, we take the following inverse source problem as an example: given $g_1$ and $g_2$ on $\Gamma$, find $p$ such that $(p,u)$ satisfies
\begin{equation}\label{BVP}
\left\{\begin{array}{l}
-\triangle u + u  = p\chi_{\Omega_0} \textrm{~in~} \Omega, \\
u = g_1 \textrm{~and~} \frac{\partial u}{\partial \mathbf{n}} = g_2  \textrm{~on~} \Gamma,
\end{array}\right.
\end{equation}
where $\Omega\subset \mathbb{R}^d$ ($d=2,3$) represents a bounded domain with a smooth boundary $\Gamma$, $\partial/\partial \mathbf{n}$ stands for the unit outward normal derivative, $\Omega_0\subset\Omega$ is known as a permissible region of the source function, and $\chi$ is the indicator function such that $\chi_{\Omega_0}(x)=1$ for $x\in \Omega_0$, while $\chi_{\Omega_0}(x)=0$, when $x\not\in \Omega_0$. Note that the framework proposed in this paper can also be applied to various linear and nonlinear inverse problems in PDEs, e.g. inverse source problems in parabolic or hyperbolic PDEs, parameter identification problems in PDEs, etc.

The variational methods of solving (\ref{BVP}) are usually classified into two groups: the boundary fitting formulation and the domain fitting formulation. For the boundary fitting formulation, we use one of the boundary conditions to form a boundary value problem, and the remaining boundary condition as the object-optimized function to determine the source term. For instance, the following formulation can be considered (\cite{Han:2006})
\begin{equation}
\min_{p} \frac{1}{2} \|u(p)-g_1\|^2_{0,\Gamma}, \label{problem1}
\end{equation}
where $u(p)$ is the weak solution in $H^1(\Omega)$ of (\ref{BVP}) with the Neumann boundary condition, and $\|\cdot\|_{0,\Gamma}$ is the standard norm of $L^2(\Gamma)$.

The Kohn-Vogelius method is certainly the most prominent domain fitting formulation for the inverse source problem (\ref{BVP}). In this approach, the following optimization problem is adopted (\cite{Afraites:2007,Song-2012}):
\begin{equation}
\min_{p} \frac{1}{2} \|u_1(p)-u_2(p)\|^2_{0,\Omega},
\label{KV}
\end{equation}
where $u_{1}, u_{2}\in H^1(\Omega)$ are the weak solutions of $-\triangle u_{1,2} + u_{1,2}  = p\chi_{\Omega_0}$
with Dirichlet and Neumann data respectively, and $\|\cdot\|_{0,\Omega}$ is the standard norm of $L^2(\Omega)$.

However, both formulations (\ref{problem1}) and (\ref{KV}) use the Neumann and Dirichlet data separately. In~\cite{Cheng:2014}, a novel coupled complex boundary method (CCBM) was introduced. The idea of CCBM is to couple the Neumann data and Dirichlet data in a Robin boundary condition, which leads to the following optimization problem
\begin{equation}
\min_{p} \frac{1}{2} \|u_{im}\|^2_{0,\Omega}.
\label{CCBM}
\end{equation}
where $u=u_{re}+i u_{im}$ ($i=\sqrt{-1}$ is the imaginary unit) solves
\begin{equation}\label{CCB}
\left\{\begin{array}{ll}
-\triangle u + u  = p \chi_{\Omega_0} & \textrm{~in~} \Omega, \\
\frac{\partial u}{\partial \mathbf{n}} + i u = g_2 + i g_1 & \textrm{~on~} \Gamma.
\end{array}\right.
\end{equation}

Obviously, all formulations (\ref{problem1}), (\ref{KV}) and (\ref{CCBM}) are still ill-posed, since a general source could not be determined uniquely by the boundary measurements, see e.g.,~\cite{Isakov:1990,Alves:2009}. Moreover, the mapping from the source function to the boundary data is a compact operator in Hilbert spaces, which implies the unboundedness of its inversion operator. Therefore, for the problem with noisy boundary data, regularization methods should be employed for obtaining stable approximate solutions. Loosely speaking, three groups of regularization methods exist: descriptive regularization methods, variational regularization methods and iterative regularization methods.

Descriptive regularization uses a priori information of the solution to overcome the ill-posedness of the original inverse problem. For inverse source problems, under the assumption of sourcewise representation of the unknown source function, the authors in~\cite{ZhangJIIP2018} combined the expanding compacts method and CCBM to propose a new efficient regularization method. However, in this paper, we are interested in a more general case that no a priori information about the solution is available.

Tikhonov regularization should be the most prominent variational regularization method. Denote $V(p)$ as the objective functional in (\ref{problem1}), (\ref{KV}) or (\ref{CCBM}).  With the
Tikhonov regularization, the original inverse source problem (\ref{BVP}) is converted to the following minimization problem:
\begin{equation}
p_\varepsilon = \mathop{\arg\min}_{p} V_\varepsilon(p), \qquad V_\varepsilon(p):=V(p) + \frac{\varepsilon}{2} \|p\|^2_{0,\Omega_0},
\label{Regularization}
\end{equation}
where $\varepsilon>0$ is a regularization parameter chosen in a special way using the noisy boundary data.
Under certain assumptions, (\ref{Regularization}) admits a unique solution
$p_\varepsilon$, which converges to the minimal norm solution of (\ref{BVP}) with the noise-free boundary data (\cite{Han:2006,Afraites:2007,Cheng:2014}).

In this paper, our focus is on the iterative regularization approaches, since, from a computational viewpoint, the iterative approach seems more attractable, especially for large-scale problems. The most famous iterative regularization approach should be the Landweber iteration, which is defined by (cf., e.g.,~\cite{engl1996regularization,Kaltenbacher-2008})
\begin{eqnarray}\label{Linear}
x_{k+1} = x_{k} - \Delta t \nabla V (p),
\end{eqnarray}
which can be viewed as a discrete analog of the following first order evolution equation
\begin{eqnarray}\label{FisrtFlow}
\dot{x}(t) = -\nabla V (p(t)),
\end{eqnarray}
where $\nabla$ denotes the gradient of $V$, and $t$ is the introduced artificial time. The formulation (\ref{FisrtFlow}) is known as the asymptotical regularization, or the Showalter's method. The regularization property of (\ref{FisrtFlow}) can be analyzed through a proper choice of the terminating time.

It is well known that the original Landweber method works quite slowly. Thus, accelerating strategies are usually adopted in practice. In recent years, there has been increasing evidence to show that the second order iterative methods exhibit remarkable acceleration properties for stably solving ill-posed problems. The most well-known methods are the Nesterov acceleration scheme (\cite{Neubauer-2017}), the $\nu$-method~\cite[\S~6.3]{engl1996regularization}, and the two-point gradient method (\cite{Hubmer-2017}). Recently, the authors in~\cite{ZhangHof2018} have established an initial theory of the second order asymptotical regularization method with fixed damping parameter for solving general linear ill-posed inverse problems. In this paper, inspired by the development of second order dynamics for accelerating the convergence of iterative regularization methods in \cite{Hubmer-2017,ZhangHof2018}, we develop a second order asymptotical regularization method for solving the inverse source problem (\ref{BVP}), i.e., we consider the second order evolution equation
\begin{eqnarray}\label{SecondFlow}
\left\{\begin{array}{ll}
\ddot{p}(t) + \eta(t) \dot{p}(t) + \nabla V (p(t)) = 0,  \\
p(0)=p_0, \quad \dot{p}(0)= \dot{p}_0,
\end{array}\right.
\end{eqnarray}
where $(p_0,\dot{p}_0) \in P\times P$ is the prescribed initial data, $\eta>0$ is the so-called damping parameter, which may or may not depend on the artificial time $t$, and $P$ is the solution space, which will be precisely defined later. It is not difficult to show that the evolution equation (\ref{SecondFlow}) with the following specific choice of discretization parameters
\begin{eqnarray*}
\left\{\begin{array}{l}
\Delta t_k= 4 \frac{(2k+2\nu-1)(k+\nu-1)}{(k+2\nu-1)(2k+4\nu-1)}, \vspace{3mm}\\
\eta_k= \frac{(k+2\nu-1)(2k+4\nu-1)(2k+2\nu-3) - (k-1)(2k-3)(3k+3\nu-1)}{4(2k+2\nu-3)(2k+2\nu-1)(k+\nu-1)},
\end{array}\right.
\end{eqnarray*}
yields the $\nu$-method. Moreover, as demonstrated in~\cite{Su-2016}, (\ref{SecondFlow}) with a special choice of damping parameter can be considered as an infinite dimensional extension of the Nesterov's scheme in the following sense.

\begin{theorem}
Let $\{p_k\}$ be the sequence, generated by the Nesterov's scheme with parameters $(\alpha,\omega)$, see (\ref{Nesterov}) for details. Then, for all fixed $T > 0$: \begin{eqnarray*}
\lim\limits_{\omega\to0} \max\limits_{0\leq k \leq T/\sqrt{\omega}} \|p_k - p(k\sqrt{\omega})\|_P =0,
\end{eqnarray*}
where $p(\cdot)$ is the solution of (\ref{SecondFlow}) with $\eta(t)=\alpha/t$.
\end{theorem}

The remainder of the paper is structured as follows: Section 2 discusses some properties of the solution of evolution equation (\ref{SecondFlow}). The convergence analysis for exact and noisy data are presented in Sections 3 and 4, respectively. Finite dimensional approximation of our method is proposed in Section 5, where we develop a novel second order iterative regularization algorithm. Some numerical examples, as well as a comparison with three existing iterative regularization methods, are presented in Section 6. Finally, concluding remarks are given in Section 7.

\section{Properties of the second order evolution equation}

For clarity, we only consider the formulation (\ref{CCBM}) in this paper. Let us first introduce the notations for the function spaces that are used in this paper. For a set $G$ (e.g., $\Omega$, $\Omega_0$ or $\Gamma$), denote by $W^{m,s}(G)$ the Sobolev space with norm $\|\cdot\|_{m,s,G}$. In particular, $L^s(G):=W^{0,s}(G)$. Moreover, $H^m(G)$ represents $W^{m,2}(G)$
with the corresponding inner product $(\cdot,\cdot)_{m,G}$ and norm $\|\cdot\|_{m,G}$. Let $\mathbf{H}^m(G)$
be the complex version of $H^m(G)$ with inner product $((\cdot,\cdot))_{m,G}$ and norm $|\|\cdot\||_{m,G}$
defined as follows: $\forall u,v \in \mathbf{H}^m(G), ((u,v))_{m,G}=(u,\bar{v})_{m,G}, |\|u\||_{m,G}=((u,u))^{1/2}_{m,G}$, where $\bar{v}$ is the conjugate complex of $v$. Denote $P= L^2(\Omega_0)$ or $H^1(\Omega_0)$ as the space for the source function $p$. Its corresponding inner product and norm are given by $(\cdot,\cdot)_{P}$ and $\|\cdot\|_{P}$, respectively.

Assume that $g_1\in H^{1/2}(\Gamma)\cap L^\infty(\Gamma)$ and $g_2\in L^{\infty}(\Gamma)$. Moreover, instead of the exact data $\{g_1,g_2\}$, we have only the noisy data $g^\delta_1, g^\delta_2\in L^{\infty}(\Gamma)$ such that
\begin{equation}\label{noiseLevel1}
\|g^\delta_1 -g_1\|_{\infty,\Gamma}\leq \delta, \quad \|g^\delta_2 -g_2\|_{\infty,\Gamma}\leq \delta,
\end{equation}
where $\delta>0$ denotes the error level of the measurement. Then, the CCBM for inverse source problem (\ref{BVP}) with noisy data $\{g^\delta_1,g^\delta_2\}$ can be formulated as
\begin{equation}
\inf_{p\in P} V(p) = \inf_{p\in P} V(p;\delta) =  \inf_{p\in P} \frac{1}{2} \|u_{im}(p)\|^2_{0,\Omega},
\label{problemFinal}
\end{equation}
where $u=u_{re}+i u_{im}$ solves
\begin{equation}\label{CCBNoise}
\left\{\begin{array}{ll}
-\triangle u + u  = p \chi_{\Omega_0} & \textrm{~in~} \Omega, \\
\frac{\partial u}{\partial \mathbf{n}} + i u = g^\delta_2 + i g^\delta_1 & \textrm{~on~} \Gamma.
\end{array}\right.
\end{equation}

Suppose that system (\ref{BVP}) has at least one solution $(p,u)$ for noise-free data and denote by $p^\dagger$ one of the solutions, i.e.
\begin{equation}\label{exactp}
p^\dagger\in \mathop{\arg\min}_{p\in P} V(p;0).
\end{equation}

\begin{proposition}
~\cite[Proposition 1]{ZhangYe2018} The Fr\'echet derivative of $V(p)$, defined in (\ref{problemFinal}),
is the imaginary part of the solution to the adjoint problem
\begin{eqnarray}\label{adjointCCBN}
\left\{\begin{array}{ll}
-\triangle w + w  = u_{im}(p) & \textrm{~in~} \Omega, \\
\frac{\partial w}{\partial \mathbf{n}} + i w = 0 & \textrm{~on~} \Gamma,
\end{array}\right.
\end{eqnarray}
where $u_{im}$ is the imaginary part of $u$, the solution of (\ref{CCBNoise}), i.e., $\nabla_p V(p) = w_{im}(p)\chi_{\Omega_0}$.
\end{proposition}

It is not difficult to show that $V''(p) q^2=\|u_{im}(q)-u_{im}(0)\|^2_{0,\Omega}$. Hence, $V(p)$ is convex.

Now we are in a position to introduce the second order asymptotical regularization for solving the inverse source problem (\ref{BVP}).

\begin{definition}
An element $p^\delta(x,T^*)\in P$ with an appropriate selected terminating time point $T^*=T^*(\delta)$ is called a second order asymptotical regularized solution if $p^\delta(x,t)$ is the solution to the following Cauchy problem
\begin{equation}\label{DDBNandDDS}
\left\{\begin{array}{ll}
\ddot{p}^\delta(x,t) + \eta(t) \dot{p}^\delta(x,t) + w_{im}(x,t) =0, & x\in\Omega_0,~ t\in (0,\infty), \\
p^\delta(x,0)=p_{0}(x), \dot{p}^\delta(x,0)=\dot{p}_{0}(x), & x\in\Omega_0,
\end{array}\right.
\end{equation}
where $w=w_{re}+i w_{im}$ is the solution of the adjoint problem with the same $t$
\begin{equation}\label{prow}
\left\{\begin{array}{ll}
-\triangle w(x,t) +   w(x,t)  = u_{im}(p^\delta(x,t)), & x\in\Omega,~ t\in (0,\infty), \\
\frac{\partial w(x,t)}{\partial \mathbf{n}} + i w(x,t) = 0, & x\in\Gamma,~ t\in (0,\infty),
\end{array}\right.
\end{equation}
and $u=u_{re}+i u_{im}$ is the solution of the BVP
\begin{equation}\label{prou}
\left\{\begin{array}{ll}
-\triangle u(x,t) +   u(x,t)  = p^\delta(x,t) \chi_{\Omega_0}, & x\in\Omega,~ t\in (0,\infty), \\
\frac{\partial u(x,t)}{\partial \mathbf{n}} + i u(x,t) = g^\delta_2(x) + i g^\delta_1(x),
& x\in\Gamma,~ t\in (0,\infty).
\end{array}\right.
\end{equation}
\end{definition}

Before presenting the solvability of system (\ref{DDBNandDDS})-(\ref{prou}), we discuss the well-posedness of the BVPs (\ref{prow}) and (\ref{prou}). For any $u,\psi\in \mathbf{H}^1(\Omega)$, define
\begin{eqnarray*}
&& a(u,\psi)= \int_\Omega \left(\nabla u \cdot \nabla \bar{\psi} +   u \bar{\psi} \right)dx + i \int_\Gamma u \bar{\psi} ds,  \\
&& f^\delta(\psi)= \int_{\Omega_0} p^\delta \bar{\psi} dx + \int_\Gamma g^\delta_2 \bar{\psi} ds + i \int_\Gamma g^\delta_1 \bar{\psi} ds.
\end{eqnarray*}
Then the weak form of the BVP (\ref{prou}) reads:
\begin{eqnarray}
\textrm{~find~} u\in \mathbf{H}^1(\Omega) \textrm{~such that~} a(u,\psi)=f^\delta(\psi), \quad \forall \psi\in \mathbf{H}^1(\Omega).
\label{weakBVP}
\end{eqnarray}

\begin{lemma}\label{Lemma_u}
(\cite{Cheng:2014}) Problem (\ref{weakBVP}) admits a unique solution $u\in \mathbf{H}^1(\Omega)$
which depends continuously on $p^\delta$, $g^\delta_1$ and $g^\delta_2$. Furthermore, a constant $C(\Omega)$ exists such that
\begin{eqnarray}
|\|u\||_{1,\Omega} \leq C(\Omega) \left( \|p^\delta\|_{0,\Omega_0} +  \|g^\delta_1\|_{0,\Gamma} +  \|g^\delta_2\|_{0,\Gamma} \right).
\label{inequality1}
\end{eqnarray}
\end{lemma}

By Lemma~\ref{Lemma_u} and the definition of $V(p)$ in (\ref{problemFinal}) and $w$ in (\ref{adjointCCBN}),
it is not difficult to prove the following lemma.

\begin{lemma}\label{LemmaV}
The following two inequalities hold for some constants $C(\Omega)$:
\begin{eqnarray}
V(p^\delta) \leq C(\Omega) \left( \|p^\delta\|^2_{0,\Omega_0} +  \|g^\delta_1\|^2_{0,\Gamma} +  \|g^\delta_2\|^2_{0,\Gamma} \right), \label{inequalityVp} \\
|\|w(p^\delta)\||_{1,\Omega} \leq C(\Omega) \left( \|p^\delta\|_{0,\Omega_0} +  \|g^\delta_1\|_{0,\Gamma} +  \|g^\delta_2\|_{0,\Gamma} \right).
\label{inequalityw}
\end{eqnarray}
\end{lemma}

\begin{theorem}\label{ThmDynamic}
For each pair $(p_{0}, \dot{p}_{0})\in P\times P$, system (\ref{DDBNandDDS})-(\ref{prou}) has a unique weak solution which depends continuously on the boundary data $\{g^\delta_1, g^\delta_2\}$.
\end{theorem}

The proof is similar to those of (a) in ~\cite[Theorem 1]{ZhangYe2018}. A sketch of the proof is given in the Appendix A.


\section{Convergence for noise-free boundary data} In this section, we investigate two models: when the damping parameter $\eta$ is fixed, and when it is time dependent. For simplicity, sometimes let $p(t)=p(\cdot,t)$.

\subsection{Case I: $\eta$ is a constant}

We first study the dynamics of the solution $p(t)\in P$ of system (\ref{DDBNandDDS})-(\ref{prou}).

\begin{lemma}\label{LemmaVanishingExact}
Let $p(x,t)$ be the solution of (\ref{DDBNandDDS})-(\ref{prou}) with the exact data $\{g_1, g_2\}$. Then, in the case $\eta\geq1$, we have
\begin{itemize}
\item[(i)] $p\in L^\infty([0,\infty), P)$.
\item[(ii)] $\dot{p}\in L^\infty([0,\infty),P) \cap L^2([0,\infty),P)$ and $\dot{p}(\cdot,t)\to0$ as $t\to\infty$.
\item[(iii)] $\ddot{p}\in L^\infty([0,\infty),P) \cap L^2([0,\infty),P)$ and $\ddot{p}(\cdot,t)\to0$ as $t\to\infty$.
\item[(iv)] $V(p(\cdot,t)) = o(t^{-1})$ as $t\to\infty$.
\end{itemize}
\end{lemma}

\begin{proof}
The proof follows the idea in \cite{Attouch-2000}.  Consider for every $t\in[0,\infty)$ the function $e(t)=e(t;p^\dagger)=\frac{1}{2} \|p(t) - p^\dagger\|^2_P$, where $p^\dagger$ is defined in (\ref{exactp}). Since $\dot{e}(t)= ( p(t) - p^\dagger, \dot{p}(t) )_P $ and $\ddot{e}(t)= \|\dot{p}(t)\|^2_P+ ( p(t) - p^\dagger, \ddot{p}(t) )_P$ for every $t\in[0,\infty)$, taking into account (\ref{SecondFlow}), we get
\begin{equation}
\label{EqError2}
\ddot{e}(t) + \eta \dot{e}(t) + ( p(t) - p^\dagger, u_{im}(p(t)) )_P = \|\dot{p}(t)\|^2_P.
\end{equation}
Here, and later on, we denote $(p, u)_P=\int_{\Omega_0} p u dx$ for $P=L^2(\Omega_0)$ and $(p, u)_P=\int_{\Omega_0} p u dx + \int_{\Omega_0} \frac{\partial p}{\partial x} \frac{\partial u}{\partial x} dx$ for $P=H^1(\Omega_0)$. Moreover, $\|u\|_P=\sqrt{(u, u)_P}$.

On the other hand, by the convexity inequality of the functional $\|u_{im}( \cdot) \|^2_P$, we derive
\begin{eqnarray}\label{BoundedSolutionProof1}
\| u_{im}(p(t)) \|^2_P= \| u_{im}(p(t)) \|^2_P- \| u_{im}(p^\dagger) \|^2_P\leq ( p(t) - p^\dagger, u_{im}(p(t)) )_P .
\end{eqnarray}

Combine (\ref{EqError2}) and the above inequality to obtain
\begin{eqnarray}
\label{ProofIneqe}
\ddot{e}(t) + \eta \dot{e}(t) + \| u_{im}(p(t)) \|^2_P\leq \|\dot{p}(t)\|^2_P
\end{eqnarray}
or, equivalently (by using the equation (\ref{DDBNandDDS})),
\begin{eqnarray}
\label{ProofIneq0}
\ddot{e}(t) + \eta \dot{e}(t) + \eta \frac{d \|\dot{p}(t)\|^2_P}{dt} + \left( \eta^2 - 1 \right) \|\dot{p}(t)\|^2_P+  \|\ddot{p}\|^2_P\leq 0.
\end{eqnarray}

By the assumption $\eta\geq 1$, we deduce that
\begin{equation}
\label{BoundedSolutionProof2}
\ddot{e}(t) + \eta \dot{e}(t) + \eta \frac{d \|\dot{p}(t)\|^2_P}{dt} \leq 0,
\end{equation}
which means that the function $t \mapsto \dot{e}(t) + \eta e(t)+ \eta \|\dot{p}(t)\|^2_P$ is monotonically decreasing. Hence a real number $C$ exists such that
\begin{equation}
\label{ProofIneq5a}
\dot{e}(t) + \eta e(t)+ \eta  \|\dot{p}(t)\|^2_P\leq C,
\end{equation}
which implies $\dot{e}(s) + \eta e(s) \leq C$. By multiplying this inequality with $e^{\eta s}$ and then integrating from 0 to $t$, we obtain the inequality
\begin{eqnarray*}
e(t)\leq e(0) e^{-\eta t} + C \left( 1 - e^{-\eta t} \right)/\eta \leq e(0) + C/\eta .
\end{eqnarray*}
Hence, $e(\cdot)$ is uniform bounded, and, consequently, $p(\cdot)\in L^\infty([0,\infty), P)$.

Now, consider the long-term behavior of $\dot{p}$. Define the Lyapunov function of the differential equation (\ref{DDBNandDDS}) by $\mathcal{E}(t) = V(p(t))+ \frac{1}{2} \|\dot{p}(t)\|^2_P$. It is not difficult to show that
\begin{equation}
\label{Decreasing1}
\dot{\mathcal{E}}(t) = - \eta \|\dot{p}(t)\|^2_P
\end{equation}
by looking at the equation (\ref{DDBNandDDS}) and the differentiation of the energy function
$\dot{\mathcal{E}}(t) = ( \dot{p}(t), \ddot{p}(t) - u_{im}(p(t)) )_P$. Hence, $\mathcal{E}(t)$ is non-increasing, and consequently, $\|\dot{p}(t)\|^2_P \leq 2\mathcal{E}(0)$. Therefore, $\dot{p}(\cdot)\in L^\infty([0,\infty), P)$.
Integrating both sides in (\ref{Decreasing1}), we obtain
\begin{eqnarray*}
\int^{\infty}_{0} \|\dot{p}(t)\|^2_Pdt \leq  \mathcal{E}(0) / \eta <  \infty,
\end{eqnarray*}
which yields $\dot{p}(\cdot) \in L^2([0,\infty),P)$  (and $\lim_{t\to\infty} \dot{p}(t) = 0$ since $\dot{p}(\cdot) \in L^\infty([0,\infty), P) \cap L^2([0,\infty),P)$).

Define
\begin{eqnarray}\label{h}
h(t) = \frac{\eta}{2} \| p(t) - p^\dagger \|^2_P+ ( \dot{p}(t), p(t) - p^\dagger )_P.
\end{eqnarray}
By elementary calculations, we derive that
\begin{eqnarray*}
\begin{array}{ll}
\dot{h}(t) = \eta ( \dot{p}(t), p(t) - p^\dagger )_P + ( \ddot{p}(t), p(t) - p^\dagger )_P + \| \dot{p}(t) \|^2_P\\ \qquad = \| \dot{p}(t) \|^2_P- ( u_{im}(p(t)), p(t) - p^\dagger )_P,
\end{array}
\end{eqnarray*}
which implies that (by noting $\dot{\mathcal{E}}(t) = - \eta \|\dot{p}(t) \|^2_P$ and the inequality (\ref{BoundedSolutionProof1}))
\begin{eqnarray*}
3\dot{\mathcal{E}}(t) + 2\eta \mathcal{E}(t)  + \eta \dot{h}(t)
= \eta \left [ 2 V(p(t)) -  (p - p^\dagger, u_{im}(p(t)))_P \right] \leq 0.
\end{eqnarray*}

Integrate the above inequality on $[0,T]$ to obtain together with the non-negativity of $\mathcal{E}(t)$
\begin{eqnarray}\label{chiIneq}
\quad \int^T_0 \mathcal{E}(t) dt \leq \frac{3}{2\eta} \left( \mathcal{E}(0) - \mathcal{E}(t) \right) - \frac{1}{2} (h(t)-h(0))
\leq \left(  \frac{3}{2\eta} \mathcal{E}(0) + \frac{1}{2} h(0) \right)- \frac{1}{2} h(t).
\end{eqnarray}

On the other hand, since both $p(t)$ and $\dot{p}(t)$ are uniform bounded, a constant $M$ exists such that $|h(t)|\leq M$. Hence, letting $T\to \infty$ in (\ref{chiIneq}), we obtain
\begin{eqnarray}\label{chiIneq2}
\int^\infty_0 \mathcal{E}(t)  dt < \infty.
\end{eqnarray}
Hence $\lim_{t\to\infty} \mathcal{E}(t) = 0$, and, consequently, $\lim_{t\to\infty} \dot{p}(t) = 0$.

Since $\mathcal{E}(t)$ is non-increasing, we deduce that
\begin{eqnarray}\label{chiIneq3}
\int^{T}_{T/2} \mathcal{E}(t) dt \geq \frac{T}{2}\mathcal{E}(T).
\end{eqnarray}
Using (\ref{chiIneq2}), the left side of (\ref{chiIneq3}) tends to 0 when $T\to\infty$, which implies that $\lim_{T\to\infty} T\mathcal{E}(T)=0$. Hence, we conclude $\lim_{T\to\infty} T V(p(T)) =0$, which yields the desired result in (iv).

Finally, let us show the long-term behavior of $\ddot{p}(t)$. Integrating the inequality (\ref{ProofIneq0}) from 0 to $T$ we obtain that there exists a real number $C'$ such that for every $t\in[0,\infty)$
\begin{eqnarray}
\label{ProofIneqEddot}
\begin{array}{l}
\qquad \dot{e}(T) + \eta e(T) + \eta \|\dot{p}(T)\|^2_P+ \left( \eta^2 - 1 \right) \int^T_0 \|\dot{p}(t)\|^2_Pdt +  \eta \int^T_0 \|\ddot{p}(t)\|^2_PdT \leq C'.
\end{array}
\end{eqnarray}
Since both $e(\cdot)$ and $\dot{e}(\cdot)$ are global bounded (note that $p(t),\dot{p}(t)\in L^\infty([0,\infty),P)$), inequality (\ref{ProofIneqEddot}) gives $\ddot{p}(t)\in L^2([0,\infty),P)$. The relations $\ddot{p}(t)\in L^\infty([0,\infty),P)$ and $\ddot{p}(t)\to0$ as $t\to\infty$ are obvious by noting assertions (i), (ii), (iv) and the connection equation (\ref{SecondFlow}).
\end{proof}

\begin{remark}
The rate $V(p(\cdot,t)) = o(t^{-1})$ as $t \to \infty$ given in Lemma \ref{LemmaVanishingExact} for the second order evolution equation (\ref{SecondFlow}) should be compared with the corresponding result for the first order method, i.e. the gradient decent methods, where one only obtains  $V(p(\cdot,t)) = \mathcal{O}(t^{-1})$ as $t \to \infty$. If we consider a discrete iterative method with the number $k$ of iterations, assertion (iv) in Lemma \ref{LemmaVanishingExact} indicates that in comparison with gradient descent methods, the second order methods (\ref{SecondFlow}) need the same computational complexity for the number $k$ of iterations, but can achieve a higher order $o(k^{-1})$ of accuracy of the objective functional.
\end{remark}

\smallskip

Now, we list the following two lemmas, which will be used in the convergence analysis of the dynamical solution $p(x,t)$.

\smallskip

\begin{lemma}\label{Opial}
(Opial lemma~\cite{Opial}) Let $P$ be a Hilbert space and $p:[0,\infty)\to P$ be a mapping such that
there exists a non-empty set $S\subset P$ which satisfies
\begin{itemize}
\item[(i)] $\forall t_n \to \infty$ with $p(t_n)\rightharpoonup \bar{p}$ weakly in $P$, we have $\bar{p}\in S$.
\item[(ii)] $\forall p^\dagger \in S$, $\lim_{t\to \infty} \|p(t) - p^\dagger\|_P$ exists.
\end{itemize}
Then, $p(t)$ weakly converges as $t\to\infty$ to some element of $S$.
\end{lemma}

\begin{lemma}\label{LemmaVarphi}
(Lemma 4.2 in~\cite{Attouch-2000}) Let $\varphi(t)\in C^1((0,\infty),[0, +\infty))$ satisfy the
inequality $\ddot{\varphi}(t) + \eta \dot{\varphi}(t) \leq  g(t)$ with $g(t)\in L^1((0,\infty), [0, +\infty))$. Then, $\dot{\varphi}_+$, the positive part of $\dot{\varphi}$, belongs to $L^1((0,\infty), [0, +\infty))$ and, as a consequence, $\lim_{t\to\infty} \varphi(t)$ exists.
\end{lemma}

\smallskip

Now, we are in the position to present the main result in this section.

\begin{theorem}\label{ConvergenceExact}
The solution $p(x,t)$ of (\ref{DDBNandDDS})-(\ref{prou}) with the exact data converges weakly in $P$ to an exact source function of inverse source problem (\ref{BVP}) as $t\to\infty$.
\end{theorem}

\begin{proof}
It suffices to check two conditions in Opial lemma. Consider a sequence $\{p(t_n)\}$ such that $p(t_n) \rightharpoonup \bar{p}$ weakly in $P$. Applying the convexity inequality to the functional $V(p)=\frac{1}{2} \|u_{im}\|^2_{0,\Omega}$ we have
\begin{eqnarray}\label{convexityIneqProof}
V(z)\geq V(p(t_n)) + (z - p(t_n), \nabla V(p(t_n)))_P, \quad \forall z\in P.
\end{eqnarray}
By using the continuity of $V(p)$, and noticing that, in the inner product $(z - p(t_n), \nabla V(p(t_n)))_P$, the two terms are, respectively, norm converging to zero and weakly convergent, we can pass to the lower limit to obtain $V(z) \geq V(p(t_n))$ for all $z\in P$. Set $z=p^\dagger$ in the above inequality, we conclude that $0=V(p^\dagger) \geq V(\bar{p})$, which implies that $\bar{p}$ is also a solution of inverse source problem (\ref{BVP}).

Now, we prove the second requirement in Lemma \ref{Opial}. It is equivalent to show that $\lim_{t\to \infty} e(t)$ exists, where $e(t)=\frac{1}{2} \|p(t) - p^\dagger\|^2_P$ is defined in the proof of Lemma \ref{LemmaVanishingExact}. From (\ref{ProofIneqe}), we deduce that
\begin{eqnarray}\label{Section4IneqProof2}
\ddot{e}(t) + \eta \dot{e}(t) \leq \|\dot{p}(t)\|^2_P.
\end{eqnarray}
Since $\dot{p}(\cdot)\in L^2([0,\infty),P)$, inequality (\ref{Section4IneqProof2}) together with Lemma (\ref{LemmaVarphi}) yields the second condition in Opial lemma. This completes the proof of the weak convergence of the dynamical solution of (\ref{SecondFlow}).
\end{proof}

\subsection{Case II: $\eta(t) = r/t$}

Now, we study the second order dynamical system (\ref{DDBNandDDS})-(\ref{prou}) with an asymptotical vanishing damping parameter of the type $\eta(t) = r/t$, i.e. we consider the following evolution equation
\begin{equation}\label{DDBNandDDSCase2}
\left\{\begin{array}{ll}
\ddot{p}(x,t) + \frac{r}{t} \dot{p}(x,t) + w_{im}(x,t) =0, & x\in\Omega_0,~ t\in (1,\infty), \\
p(x,1)=p_{0}(x), \dot{p}(x,1)=\dot{p}_{0}(x), & x\in\Omega_0,
\end{array}\right.
\end{equation}
where $w=w_{re}+i w_{im}$ is the solution of the adjoint problem (\ref{prow}) with the same $t$. As discussed in Section 1, this is a particularly interesting case as the second order flow (\ref{DDBNandDDSCase2}) yields a continuous version of Nesterov's scheme, which has a higher order of convergence rate for the residual functional, i.e. $V(p(\cdot,t))=\mathcal{O}(k^{-2})$ for $r=3$ and $V(p(\cdot,t))=o(k^{-2})$ for $r>3$ (\cite{Attouch-2016}).

\begin{remark}
We shift the initial time point from 0 to 1 for the regularity of the term $r/t$. Otherwise, one can use $r/(t+1)$ instead of $r/t$ in (\ref{DDBNandDDSCase2}).
\end{remark}

For proving the following assertions, we introduce the anchored energy function
\begin{equation}
\label{anchoredEnergy}
\qquad \mathcal{E}_{\lambda} (t) = t^2 V(p(t)) + \frac{1}{2} \| \lambda(p(t) - p^\dagger) + t \dot{p}(t)\|^2_P + \frac{\lambda(r-1-\lambda)}{2} \| p(t) - p^\dagger\|^2_P,
\end{equation}
where the exact source $p^\dagger$ is given in (\ref{exactp}). For $r\geq3$, using the convexity inequality $0=V(p^\dagger)\geq V(p) + ( \nabla V(p), p^\dagger - p)_P$ for all $p\in P$ and (\ref{DDBNandDDSCase2}), it is not difficult to show that
\begin{eqnarray}
\label{anchoredEnergy2}
\dot{\mathcal{E}}_{\lambda} (t) \leq -(\lambda-2)tV(p(t)) - (r-1-\lambda)t \| \dot{p}(t)\|^2_P.
\end{eqnarray}
Hence, for $r\geq3$ and $\lambda\in[2,r-1]$, $\mathcal{E}_{\lambda} (t)$ is non-increasing.

Now, we are in position to derive similar results to those in Section 3.1.

\begin{lemma}\label{LemmaVanishingExactCase2}
Let $p(x,t)$ be the solution of (\ref{DDBNandDDSCase2}) with the exact data. Then, $\dot{p}\in L^\infty([1,\infty),P) \cap L^2([1,\infty),P)$ and $\dot{p}(\cdot,t)\to0$ as $t\to\infty$. Moreover, $V(p(\cdot,t)) = \mathcal{O}(t^{-2})$ as $t\to\infty$.
\end{lemma}

\begin{proof}
This proof uses the technique in~\cite{Attouch2018}. Consider the Lyapunov function of (\ref{DDBNandDDSCase2}) by $\mathcal{E}(t) = \frac{1}{2} \|\dot{p}(t)\|^2_P+ V(p(t))$. It is easy to show that
\begin{equation}\label{LyapunovDerivative2}
\dot{\mathcal{E}}(t)= - \frac{r}{t} \|\dot{p}(t)\|^2_P \leq 0.
\end{equation}
Hence, $\mathcal{E}(t)$ is non-increasing, and $\mathcal{E}(\infty) := \lim_{t\to \infty} \mathcal{E}(t)$ exists by noting that $\mathcal{E}(t)\geq0$ for all $t$. Furthermore, by $\|\dot{p}(t)\|^2_P\leq 2\mathcal{E}(t)\leq 2\mathcal{E}(1)$ we conclude the uniform boundedness of $\dot{p}(\cdot)$.

Integrating both sides in (\ref{LyapunovDerivative2}), we obtain
\begin{eqnarray*}
\int^{\infty}_{1} \|\dot{p}(t)\|^2_Pdt \leq \int^{\infty}_{1} t\|\dot{p}(t)\|^2_Pdt \leq  \mathcal{E}(1) /r <  \infty,
\end{eqnarray*}
which yields $\dot{p}(\cdot) \in L^2([1,\infty),P)$. Now, consider the function $e(t)=\frac{1}{2} \|p(t) - p^\dagger\|^2_P$. Using the local convexity of $V(\cdot)$ and the equation (\ref{SecondFlow}), similar to (\ref{ProofIneqe}), it is not difficult to obtain
\begin{equation}
\label{ProofIneqNestorovE}
\ddot{e}(t) + \frac{r}{t} \dot{e}(t) + V(p(t)) \leq  \|\dot{p}(t)\|^2_P.
\end{equation}
Divide this expression by $t$ to obtain
\begin{eqnarray*}
\frac{1}{t} \ddot{e}(t) + \frac{r}{t^2} \dot{e}(t) + \frac{1}{t} \mathcal{E}(t) \leq \frac{3}{2t} \|\dot{p}(t)\|^2_P,
\end{eqnarray*}
Integrating above inequality from $1$ to $t$ and using integration by parts for $\ddot{e}(t)$, we obtain
\begin{equation}
\label{ProofIneq1}
\int^t_{1} \frac{\mathcal{E}(\tau)}{\tau} d\tau \leq  \dot{e}(1) - \frac{\dot{e}(t)}{t} - (r+1) \int^t_{1} \frac{\dot{e}(\tau)}{\tau^2} d\tau + \frac{3}{2} \int^t_{1} \frac{\|\dot{p}(\tau)\|^2_P}{\tau}  d\tau .
\end{equation}

On one hand, using the integration by parts and the positivity of functional $e(\cdot)$, we have
\begin{equation}
\label{ProofIneq2}
\int^t_{1} \frac{\dot{e}(\tau)}{\tau^2} d\tau = \frac{e(t)}{t^2} - e(1) + 2 \int^t_{1}  \frac{e(\tau)}{\tau^3} d\tau \geq - e(1) .
\end{equation}

On the other hand, relation (\ref{LyapunovDerivative2}) gives
\begin{equation}
\label{ProofIneq3}
 \int^t_{1} \frac{\|\dot{p}(\tau)\|^2_P}{\tau}  d\tau = \frac{\mathcal{E}(1) - \mathcal{E}(t)}{r}.
\end{equation}

Combine (\ref{ProofIneq1})-(\ref{ProofIneq3}) to get
\begin{eqnarray}
\label{ProofIneq4}
\int^t_{1} \frac{\mathcal{E}(\tau)}{\tau} d\tau \leq \dot{e}(1) - \frac{\dot{e}(t)}{t} + (r+1) e(1) + \frac{3(\mathcal{E}(1) - \mathcal{E}(t))}{2r} = C(1) - \frac{\dot{e}(t)}{t} - \frac{3\mathcal{E}(t)}{2r},
\end{eqnarray}
where $C(1)= \dot{e}(1) + (r+1) e(1) + \frac{3\mathcal{E}(1)}{2r}$ collects the constant terms. For any $T\geq t> 1$, we have
\begin{equation}
\label{ProofIneq5}
\mathcal{E}(T) \int^t_{1} \frac{1}{\tau} d\tau + \frac{3\mathcal{E}(T)}{2r} \leq C(1) - \frac{\dot{e}(t)}{t}
\end{equation}
by noting the non-increasing of Lyapunov function $\mathcal{E}(t)$. Rewrite (\ref{ProofIneq5}) as $\mathcal{E}(T) \left( \ln(t)+ \frac{3}{2r} \right) \leq C(1) - \frac{\dot{e}(t)}{t}$, and then integrate it from $t=1$ to $t=T$ to have
\begin{eqnarray}
\label{ProofIneq6}
\mathcal{E}(T) \left( T \ln(T) + 1 - T + \frac{3}{2r} (T - 1) \right) \leq C(1)(T - 1) - \int^T_{1} \frac{\dot{e}(t)}{t} dt.
\end{eqnarray}

Moreover, using the integration by parts and the positivity of functional $e(\cdot)$, we have
\begin{equation}
\label{ProofIneq7}
\int^T_{1} \frac{\dot{e}(t)}{t} d\tau = \frac{e(T)}{T} - e(1) + \int^T_{1}  \frac{e(t)}{t^2} dt \geq - e(1).
\end{equation}

By combining (\ref{ProofIneq6}) and (\ref{ProofIneq7}), we deduce that
\begin{equation}
\label{ProofIneq8}
\mathcal{E}(T) \left( T \ln(T) + C_1 T + C_2 \right) \leq  C(1) T + C_3,
\end{equation}
where $C_1=\frac{3}{2r}-1$, $C_2=1-3/(2r)$ and $C_3=e(1)- C(1)$ are three constants.

Inequality (\ref{ProofIneq8}) immediately yields $\mathcal{E}(\infty)\leq0$. By the non-negativity of Lyapunov function $\mathcal{E}(\cdot)$, we conclude $\mathcal{E}(\infty)=0$, which implies that both $\dot{p}(T)$ and $V(p(T))$ converge to 0 in $P$ when $T\to\infty$.

Finally, let us show the convergence rate of $V(p(t))$. Set $\lambda=r-1$ in (\ref{anchoredEnergy}) to obtain $t^2 V(p(t))\leq \mathcal{E}_{r-1} (t)$. Since $\mathcal{E}_{r-1} (t)$ is non-increasing, we conclude that $V(p(t))\leq \mathcal{E}_{r-1} (1)/t^2$.

\end{proof}

\smallskip

\begin{lemma}\label{LemmaVarphi2}
(Lemma 5.9 in~\cite{Attouch2018}) Let $\varphi(t)\in C^1((1,\infty),[0, +\infty))$ satisfy the
inequality $t\ddot{\varphi}(t) + r \dot{\varphi}(t) \leq  g(t)$ with $r\geq1$ and $g(t)\in L^1((1,\infty), [0, +\infty))$. Then, $\dot{\varphi}_+$, the positive part of $\dot{\varphi}$, belongs to $L^1((1,\infty), [0, +\infty))$ and, as a consequence, $\lim_{t\to\infty} \varphi(t)$ exists.
\end{lemma}

\begin{theorem}
\label{NesterovWeak}
The solution $p(x,t)$ of (\ref{DDBNandDDSCase2}) with $r>3$ converges weakly to an exact source function of inverse source problem (\ref{BVP}) as $t\to\infty$.
\end{theorem}

\begin{proof}
Set $\lambda=2$ in (\ref{anchoredEnergy}) to derive $\| p(t) - p^\dagger\|^2_P \leq  \frac{\mathcal{E}_{2} (t)}{r-3} \leq \frac{\mathcal{E}_{2} (1)}{r-3}$, which yields the uniform boundedness of $p(t)$. Furthermore, we have
\begin{equation}
\label{ProofIneq10}
\dot{\mathcal{E}}_{2} (t) \leq  - (r-3)t \| \dot{p}(t)\|^2_P.
\end{equation}
Integrating (\ref{ProofIneq10}) from 1 to $T$, and recalling that $\mathcal{E}_{2}(t)$ is non-negative, we obtain
\begin{equation}
\label{ProofIneq11}
\int^T_1 t \| \dot{p}(t)\|^2_P dt \leq \mathcal{E}_{2}(1) / (r-3).
\end{equation}
Let $T\to\infty$ to conclude $t \| \dot{p}(t)\|^2_P\in L^1((1,\infty), [0, \infty))$. Recall from (\ref{ProofIneqNestorovE}) to obtain $t \ddot{e}(t) + r \dot{e}(t)  \leq  t \|\dot{p}(t)\|^2_P$. From Lemma \ref{LemmaVarphi2}, and note that $t \|\dot{p}(t)\|^2_P$ is integrable on $[1,\infty)$, the limit $\lim_{t\to\infty} e(t)$ exists. This gives the second hypothesis in Opial's Lemma. The first one was established in Lemma \ref{LemmaVanishingExactCase2}, i.e. $V(p(t))\to0$ as $t\to\infty$. This completes the proof by using the Opial's Lemma \ref{Opial}.

\end{proof}

\begin{remark}\label{RemarkWeakConv}
(a) In Theorems \ref{ConvergenceExact} and \ref{NesterovWeak}, we only obtain the weak convergence for both fixed and dynamic damping parameters. One way to obtain the strong convergence result is to include a regularization term $\epsilon(t) p(x,t)$ in the evolution equation (\ref{SecondFlow}) with a specially chosen dynamic regularization parameter $\epsilon(t)$, see \cite{ZhangYe2018} for details. However, the numerical results in Section 6 show that our method works much better than this method in terms of accuracy and speed.

(b) Let $\Pi^h$ be any project operator, acting from $P$ into a finite element space $P^h$ for a fixed triangulation $\mathcal{T}^h$. Then, we have the strong convergence $\Pi^h p^{\delta}(\cdot,t)\to \Pi^h p^{\dagger}(\cdot)$ as $t\to\infty$ in $P^h$, since strong convergence and weak convergence coincide in any finite dimensional/element space. This fact will be used in Theorem \ref{ConvFE} about the strong convergence of the finite element solution.
\end{remark}

\section{Convergence for noisy data}

In this section, we investigate the regularization property of the dynamic solution $p^{\delta}(\cdot,t)$ of (\ref{DDBNandDDS})-(\ref{prou}), equipped with some appropriate selection rules of the terminating time $T^*$.

\begin{proposition}\label{TransformErrorLevel}
There exists a constant $C_0(\Omega)$, depending only on the geometry of the domain $\Omega$,
such that $\|u_{im}(p^\dagger)\|_{0,\Omega} \leq C_0(\Omega) \delta$, where $u=u_{re}+i u_{im}$
 solves (\ref{CCBNoise}) and $p^\dagger$ is defined in (\ref{exactp}). Consequently, we have
 $V(p^\dagger)\leq C^2_0 \delta^2$. If $\Omega$ is a ball in $\mathbb{R}^d$ centered at 0 with
 radius $R$ or an annulus in $\mathbb{R}^d$ centered at 0 with radius $R$ and $r(<R)$, we have
\begin{equation}\label{C0}
C_0(\Omega) = \max(d,R) (2\pi)^{1/2}.
\end{equation}
\end{proposition}

\begin{proof}
Denote by $\tilde{u}$ the weak solution of (\ref{CCBNoise}) with the exact source term $p^\dagger$. Define $v:= u-\tilde{u}$. Then $v$ satisfies
\begin{equation}\label{CCB4}
\left\{\begin{array}{ll}
-\triangle v +   v  = 0 & \textrm{~in~} \Omega, \\
\frac{\partial v}{\partial \mathbf{n}} + i v = (g^\delta_2-g_2) + i (g^\delta_1-g_1) & \textrm{~on~} \Gamma.
\end{array}\right.
\end{equation}

The weak form of the above BVP (\ref{CCB4}) reads:
\begin{equation}
\textrm{~find~} v\in \mathbf{H}^1(\Omega) \textrm{~such that~} a(v,\psi)=\tilde{f}^\delta(\psi), \quad \forall \psi\in \mathbf{H}^1(\Omega),
\label{weakBVP4}
\end{equation}
where $\tilde{f}^\delta(\psi)= \int_\Gamma (g^\delta_2-g_2) \bar{\psi} ds + i \int_\Gamma (g^\delta_1-g_1) \bar{\psi} ds$. Denote by $v_{re}$ and $v_{im}$ the real and imaginary parts of $v$, respectively. Obviously, $v_{im}\equiv u_{im}$ by noting $\tilde{u}_{im}=0$. Furthermore, if one separates the real and imaginary parts of problem (\ref{CCB4}), the real part $v_{re}$ of $v$ satisfies
\begin{eqnarray*}
\left\{\begin{array}{ll}
-\triangle v_{re} +   v_{re}  = 0 & \textrm{~in~} \Omega, \\
\frac{\partial v_{re}}{\partial \mathbf{n}} - v_{im} = g^\delta_2-g_2 & \textrm{~on~} \Gamma,
\end{array}\right.
\end{eqnarray*}
whose weak form is
\begin{equation}
\int_\Omega \left(\nabla v_{re} \cdot \nabla \psi +   v_{re} \psi \right)dx = \int_\Gamma (g^\delta_2-g_2) \psi ds + \int_\Gamma v_{im} \psi ds, \quad \forall \psi\in H^1(\Omega).
\label{weakReal}
\end{equation}
The imaginary part $v_{im}$ of $v$ satisfies
\begin{eqnarray*}
\left\{\begin{array}{ll}
-\triangle v_{im} +   v_{im}  = 0 & \textrm{~in~} \Omega, \\
\frac{\partial v_{im}}{\partial \mathbf{n}} + v_{re} = g^\delta_1-g_1 & \textrm{~on~} \Gamma,
\end{array}\right.
\end{eqnarray*}
whose weak form is
\begin{equation}
\int_\Omega \left(\nabla v_{im} \cdot \nabla \psi +   v_{im} \psi \right)dx = \int_\Gamma (g^\delta_1-g_1) \psi ds - \int_\Gamma v_{re} \psi ds, \quad \forall \psi\in H^1(\Omega).
\label{weakImaginary}
\end{equation}

Set $\psi=v_{re}$ in (\ref{weakReal}) and $\psi=v_{im}$ in (\ref{weakImaginary}), and then add these two equations together to obtain
\begin{eqnarray*}
\|v_{re}\|^2_{1,\Omega} + \|v_{im}\|^2_{1,\Omega} =  \int_\Gamma (g^\delta_2-g_2) v_{re} ds + \int_\Gamma (g^\delta_1-g_1) v_{im} ds,
\end{eqnarray*}
which implies
\begin{equation}
|\|v\||^2_{1,\Omega} \leq \delta\int_\Gamma  \left( |v_{re}| + |v_{im}| \right) ds.
\label{IneqV}
\end{equation}

On the other hand, if $\Omega$ is a ball/annulus in $\mathbb{R}^d$ centered at 0 with radius $R$ (and $r$), it holds (\cite{Motron-2002})
\begin{equation}\label{ImbeddingConstant}
\int_{\Gamma} |u(s)| ds \leq \frac{d}{R} \int_{\Omega} |u(x)| dx + \int_{\Omega} |\nabla u(x)| dx
\end{equation}
for all $u\in W^{1,1}(\Omega)$. Then, by inequality (\ref{ImbeddingConstant}) and the Cauchy-Schwarz inequality $\int_{\Omega} |u(x)| dx\leq R \pi^{1/2} \|u\|_{0,\Omega}$, we deduce that for $k=re$ or $im$
\begin{equation}\label{ImbeddingConstantProof}
\int_{\Gamma} |v_k| ds \leq  d \pi^{1/2} \|v_k\|_{0,\Omega} + R \pi^{1/2} \|\nabla v_k\|_{0,\Omega} \leq \max(d,R) \pi^{1/2} \|v_k\|_{1,\Omega}.
\end{equation}

Combine (\ref{IneqV}), (\ref{ImbeddingConstantProof}), and the inequality $\|v_{re}\|_{1,\Omega}+\|v_{im}\|_{1,\Omega} \leq \sqrt{2} |\|v\||_{1,\Omega}$ to obtain $$\|u_{im}(p^\dagger)\|_{0,\Omega} = \|v_{im}\|_{0,\Omega} \leq |\|v\||_{1,\Omega} \leq \max(d,R) (2\pi)^{1/2} \delta,$$ which yields the required result. For the general smooth bounded domain, the proposition can be proven by using the Sobolev trace embedding inequality (with the constant $S$)
\begin{equation}\label{ImbeddingConstantS}
S \int_{\Gamma} |u(s)| ds \leq \int_{\Omega} |u(x)| + |\nabla u(x)| dx.
\end{equation}
\end{proof}

\begin{remark}\label{ImbeddingConstantRek}
The best (largest) embedding constant in (\ref{ImbeddingConstantS}) equals
\begin{equation}\label{ImbeddingInf}
S = \inf_{u\in W^{1,1}(\Omega)\ W^{1,1}_0(\Omega)} \frac{\int_{\Omega} |u(x)| + |\nabla u(x)| dx}{\int_{\Gamma} |u(s)| ds}.
\end{equation}
The extrema of (\ref{ImbeddingInf}) exists as the the embedding (\ref{ImbeddingConstantS}) is compact, cf. \cite{Bonder2001}. To the best of our knowledge, the rigorous lower bounds of $S$, hence the value of $C_0(\Omega)$ in Proposition \ref{TransformErrorLevel}, for general smooth domain $\Omega$ is still open. Alternatively, one can estimate the value of $S$ by numerically solving the following non-linear eigenvalue problem
\begin{equation}\label{Su}
\left\{\begin{array}{ll}
\textrm{div}\left( \frac{\nabla u}{|\nabla u|} \right)  = 1 & \textrm{~in~} \Omega, \\
\frac{\partial u}{\partial \mathbf{n}} = \lambda |\nabla u| & \textrm{~on~} \Gamma,
\end{array}\right.
\end{equation}
by noting that the extrema in (\ref{ImbeddingConstantS}) can be assumed positive, see e.g., \cite{Tolksdorf-1984,Vazquez-1984}.
\end{remark}

\begin{proposition}
\label{LimitNoisyData}
Let $p^\delta(x,t)$ be the dynamic solution of (\ref{DDBNandDDS})-(\ref{prou}) with the fixed damping parameter $\eta\geq1$ or $\eta(t)=r/t$ ($r>3$). Then, $\lim_{t\to\infty} V(p^\delta(x,t)) \leq C^2_0\delta^2$, where $V(\cdot)$ is defined in (\ref{problemFinal}).
\end{proposition}

The proof of the above proposition is provided in the Appendix B. Now, we discuss the method of selecting the terminating time $T^*$. In this work, we consider the following two discrepancy functions:
\begin{itemize}
\item The Morozov's conventional discrepancy function:
\begin{equation}
\chi(T)= \|u_{im}(p^\delta(x,T))\|_{0,\Omega}- C_0 \tau \delta,
\label{Morozov}
\end{equation}
where $u=u_{re}+i u_{im}$ is the solution of (\ref{prou}) with noisy data, and $\tau$ is a fixed positive number.
\item The total energy discrepancy function:
\begin{equation}
\chi_{TE}(T)= V(p^\delta(x,T)) + \|\dot{p}^\delta(x,T)\|^2_P- C^2_0 \tau^2 \delta^2,
\label{TotalEnergy}
\end{equation}
where $V(p^\delta)=\|u_{im}(p^\delta)\|^2_{0,\Omega}$.
\end{itemize}

\smallskip

\begin{lemma}\label{Rootdiscrepancy}
Under the assumption $\tau>1$, the following two assertions hold.\\
(i) If $\|u_{im}(p_0)\|_{0,\Omega}\geq C_0 \tau\delta$, then $\chi(T)$ has at least one root.\\
(ii) If $V(p_0) + \|\dot{p}_0\|^2_P\geq C^2_0 \tau \delta^2$, then $\chi_{TE}(T)$ has a unique solution.
\end{lemma}

\smallskip

\begin{proof}
The continuity of $\chi(T)$ and $\chi_{TE}(T)$ are obviously according to Lemma \ref{LemmaV} and Theorem \ref{ThmDynamic}. From Proposition \ref{LimitNoisyData} and the assumption of the lemma, we conclude that
\begin{eqnarray}\label{TwoLimits}
\lim_{T\to \infty} \chi(T) \leq C_0(1-\tau)\delta  <0  \textrm{~and~} \lim_{T\to \infty} \chi_{TE}(T) \leq C^2_0(1-\tau^2)\delta^2  <0,
\end{eqnarray}
and $\chi(0) = \|u_{im}(p_0)\|_{0,\Omega} - C_0 \tau \delta >0$ and $\chi_{TE}(0) = V(p_0) + \|\dot{p}_0\|^2_P - C^2_0 \tau^2 \delta^2 >0$, which implies the existence of the root of $\chi(T)$ and $\chi_{TE}(T)$.

The non-growing of $\chi_{TE}(T)$ is straightforward according to $\dot{\chi}_{TE} = - \eta \|\dot{p}^\delta\|^2_P$ for the fixed damping parameter and $\dot{\chi}_{TE} = - \frac{r}{T} \|\dot{p}^\delta\|^2_P$ for the dynamic damping parameter.

Finally, let us show that $\chi_{TE}(T)$ has a unique solution. We prove this by contradiction. Since $\chi_{TE}(T)$ is a non-increasing function, a number $T_0$ exists so that $\chi_{TE}(T)=0$ for $T\in [T_0,T_0+\varepsilon]$ with some positive $\varepsilon>0$. This means that $\dot{\chi}_{TE}(T)=- \eta \|\dot{p}^\delta\|^2_P\equiv0$ (or $\dot{\chi}_{TE}(T)=- \frac{r}{t} \|\dot{p}^\delta\|^2_P\equiv0$) in $(T_0,T_0+\varepsilon)$. Hence, $\ddot{p}^\delta\equiv0$ in $(T_0,T_0+\varepsilon)$. Using the equation (\ref{SecondFlow}) we conclude that for all $T>T_0$: $p^\delta(T)\equiv p^\delta(T_0)$. Since $\chi_{TE}(T_0)=0$, we obtain that $\chi_{TE}(T)\equiv0$ for $T>T_0$, which implies that $\lim \limits_{T\to \infty} \chi_{TE}(T) = 0$. This contradicts the fact in (\ref{TwoLimits}).
\end{proof}

\begin{remark}\label{Remarktau}
It should be noted that Lemma \ref{Rootdiscrepancy} may still hold in the case $\tau\leq1$. In many situations, e.g. for our numerical examples in Section 6, a small value of $\tau$ offers a better result, provided the existence of the root of $\chi$ or $\chi_{TE}$.
\end{remark}


\smallskip

\begin{theorem}\label{ThemRegu1}
(Convergence for noisy data) Let $p^\delta(x,t)$ be the dynamic solution of (\ref{DDBNandDDS})-(\ref{prou}). Then, if the terminating time point $T^*$ is selected as the root of $\chi(T)$ or $\chi_{TE}(T)$, $p^\delta(x,T^*(\delta))$ converges weakly to $p^\dagger(x)$ in $P$ as $\delta\to0$.
\end{theorem}

\begin{proof}
We use the technique from \cite[Theorem 2.4]{Scherzer-1995}. Let $\{\delta_n\}$ be a sequence converging to 0 as $n\to\infty$, and let $\{ g^{\delta_n}_1, g^{\delta_n}_2\}$ be a corresponding sequence of noisy data with $\|g^{\delta_n}_1 - g_1\|_{0,\Gamma}\leq \delta_n$ and $\|g^{\delta_n}_2 - g_2\|_{0,\Gamma}\leq \delta_n$. For a triple $(\delta_n,g^{\delta_n}_1,g^{\delta_n}_2)$, denote by $T^*_n = T^*(\delta_n)$ the corresponding terminating time point determined from the generalized discrepancy principles $\chi(T)=0$ or $\chi_{TE}(T)=0$.

Two possible cases exist. (i) $T^*_n$ has a finite accumulation point $T^*$. (ii) $T^*_n\to\infty$ as $\delta_n\to0$. For the case (i), without loss of generality we can assume that $T^*_n=T^*$ for all $n\in \mathbf{N}$. Hence, from the definition of $T^*_n$ it follows that
\begin{eqnarray}
\label{ThmProofIneq1}
\|u_{im}(p^{\delta_n}(\cdot,T^*_n))\|_{0,\Omega} \leq C_0\tau \delta_n.
\end{eqnarray}
Since $p^{\delta_n}(\cdot,T^*_n)$ depends continuously on $\{ g^{\delta_n}_1, g^{\delta_n}_2\}$ when $T^*_n$ is fixed, we have
\begin{eqnarray}
\label{ThmProofLimit1}
\quad p^{\delta_n}(\cdot,T^*_n) \to p(\cdot,T^*), \quad \|u_{im}(p^{\delta_n}(\cdot,T^*_n))\|_{0,\Omega} \to \|u_{im}(p(\cdot,T^*))\|_{0,\Omega}. \quad n\to \infty,
\end{eqnarray}
where $p(\cdot,t)$ denotes the dynamic solution of (\ref{DDBNandDDS})-(\ref{prou}) with noise-free data. Letting $n\to\infty$ in (\ref{ThmProofIneq1}) yields $\|u_{im}(p(\cdot,T^*))\|_{0,\Omega} = 0$. Thus, $p(x,T^*)=p^\dagger(x)$, a solution of (\ref{BVP}), and with (\ref{ThmProofLimit1}) we obtain the strong convergence: $p(\cdot,T^*_n)\to p^\dagger(\cdot)$ in $P$ as $n\to\infty$.

Now, consider the case (ii). According to the continuity of $p^{\delta_n}(\cdot,t)$, for any positive $\varepsilon_0$ and $T^*_n$, there exists a point $T^*<T^*_n$ such that
\begin{eqnarray}
\label{ThmProofT}
\| p^{\delta_n}(\cdot,T^*_n) - p^{\delta_n}(\cdot,T^*) \|_P \leq \varepsilon_0.
\end{eqnarray}
On the other hand, for any $q(\cdot) \in P$,
\begin{eqnarray*}
&& | ( p^{\delta_n}(\cdot,T^*_n) - p^\dagger(\cdot) , q(\cdot) )_P |  \leq  \\ && \quad
| ( p^{\delta_n}(\cdot,T^*_n) - p^{\delta_n}(\cdot,T^*) , q(\cdot) )_P | + | ( p^{\delta_n}(\cdot,T^*) - p(\cdot,T^*) , q(\cdot) )_P | + | p(\cdot,T^*) - p^\dagger(\cdot), q(\cdot) )_P |.
\end{eqnarray*}
By inequality (\ref{ThmProofT}) and the weak convergence of $p(\cdot,t)$, one can fix $T^*$ so large that both inequalities $| ( p^{\delta_n}(\cdot,T^*_n) - p^{\delta_n}(\cdot,T^*) , q(\cdot) )_P | \leq \varepsilon/3$ and $| p(\cdot,T^*) - p^\dagger(\cdot), q(\cdot) )_P | \leq \varepsilon/3$ hold. Now that $T^*$ is fixed, we can apply the result of case (i) to conclude that a positive number $n_1=n_1(T^*)$ exists such that for any $n\geq n_1$: $| ( p^{\delta_n}(\cdot,T^*) - p(\cdot,T^*) , q(\cdot) )_P | \leq \varepsilon/3$. Combine the above inequalities to obtain $| ( p^{\delta_n}(\cdot,T^*_n) - p^\dagger(\cdot) , q(\cdot) )_P | \leq \varepsilon$ for all $n\geq n_1$. Since $\varepsilon$ is arbitrary, we complete the proof.
\end{proof}


\section{Full discretization and a novel iterative regularization algorithm}

\subsection{Space discretization}
 Following~\cite{Johnson-2009}, we discretize the bounded domain $\Omega$ by
mesh $\mathcal{T}$ using non-overlapping triangles/tetrahedrons $\{\bigtriangleup_\mu\}^M_{\mu=1}$. We associate the mesh $\mathcal{T}$ with
the mesh function $h(x)$, which is a piecewise-constant function such that
$h(x)\equiv \ell(\bigtriangleup_\mu)$ for all $x\in\bigtriangleup_\mu$,
where $\ell(\bigtriangleup_\mu)$ is the longest side of $\bigtriangleup_\mu\in \mathcal{T}$. Define the mesh scale size as $h:=\max_{x\in \Omega} h(x)$.
Let $r(\bigtriangleup_\mu)$ be the radius of the maximal circle/ball contained in the triangle/tetrahedron $\bigtriangleup_\mu$. We make the following shape regularity assumption for every element
$\bigtriangleup_\mu\in \mathcal{T}$: $c_1\leq \ell(\bigtriangleup_\mu) \leq c_2 r(\bigtriangleup_\mu)$,
where $c_1$ and $c_2$ are two positive constants. Now, we introduce the finite element space
\begin{equation}\label{ElementSpace}
\Psi^h = \left\{v\in C(\Omega):~ v\in \mathcal{P}_1(\bigtriangleup_\mu)  \textrm{~for all~}
\bigtriangleup_\mu\in \mathcal{T} \right\},
\end{equation}
where $\mathcal{P}_1(\bigtriangleup_\mu)$ denotes the set of all linear continuous
functions on $\bigtriangleup_\mu$.

Denote $\mathbf{\Psi}^h:= \Psi^h \oplus i\Psi^h$. Then, $\mathbf{\Psi}^h$ is a finite element
subspace of $\mathbf{H}^1(\Omega)$, and the finite element approximation of the BVP (\ref{weakBVP}) is as follows:
\begin{eqnarray}
\textrm{~find~} u^h\in \mathbf{\Psi}^h \textrm{~such that~} a(u^h,\psi^h)
=f^\delta(\psi^h), \quad \forall \psi^h\in \mathbf{\Psi}^h.
\label{discreteBVP}
\end{eqnarray}

The problem (\ref{discreteBVP}) admits a unique solution $u^h\in \mathbf{\Psi}^h$
according to Lemma~\ref{Lemma_u}. Similar to those in~\cite{Cheng:2014}, it is not difficult to derive the following a priori finite element error estimates.

\begin{theorem}\label{FiniteElement_u}
Let $u\in \mathbf{H}^1(\Omega)$ be the solution of the problem
(\ref{weakBVP}) and $u^h\in \mathbf{\Psi}^h$ be the finite element solution of problem (\ref{discreteBVP})
respectively. Then, for any $p(\cdot,t)\in L^2((t_0,\infty),P)$
and almost every $t>0$
\begin{eqnarray*}
|\|u^h(p(\cdot,t))-u(p(\cdot,t))\||_{1,\Omega} \leq C(\Omega) h \left( \|p(\cdot,t)\|_{0,\Omega_0}
+ \|g^\delta_1\|_{0,\Gamma} +  \|g^\delta_2\|_{0,\Gamma} \right).
\end{eqnarray*}
\end{theorem}

Note that, in this section, we set $t_0=0$ or 1, corresponding to the model (\ref{DDBNandDDS}) with different damping parameter $\eta(t)=const.$ or $r/t$. Now we are in a position to discretize the second order evolution equation (\ref{DDBNandDDS}). For this purpose, set $P^h=P\cap \Psi^h$ and the orthogonal projection operator $\Pi^h: P\to P^h$
\begin{equation}
(\Pi^h p,q^h)_{k,\Omega_0} = (p,q^h)_{k,\Omega_0}, \quad \forall p\in P, q^h\in P^h,~ k=0, 1.
\label{projectionEquation0}
\end{equation}
Then for all $p\in H^{k+1}(\Omega_0)$ \cite[Theorem 10.3.8]{Atkinson:2009}:
\begin{equation}
\|\Pi^h p - p\|_{m,\Omega_0}\leq C(\Omega) h^{k+1-m} |p|_{k+1,\Omega_0}, \quad m=0,1.
\label{projectionInequation}
\end{equation}

Introduce a discrete optimization problem
\begin{equation}
\min_{p\in P^h} V_h (p) = \min_{p\in P^h} \frac{1}{2} \|u^h_{im}(p)\|^2_{0,\Omega},
\label{problemFinalDiscretized}
\end{equation}
where $u^h=u^h_{re}+i u^h_{im} \in \mathbf{\Psi}^h$ is the weak solution of the problem
(\ref{discreteBVP}), and a semi-discretized second order flow
\begin{equation}\label{DDBNandDDS_semiDiscretized}
\left\{\begin{array}{ll}
\ddot{p}^{\delta,h} (x,t) + \eta(t) \dot{p}^{\delta,h} (x,t) + w^h_{im}(x,t) =0, & x\in \Omega_0,~t\in(t_0,\infty), \\
p^{\delta,h} (x,t_0)=p^h_{0}, \dot{p}^{\delta,h} (x,t_0)=\dot{p}^h_{0}, & x\in \Gamma,~t\in(t_0,\infty),
\end{array}\right.
\end{equation}
where $p^h_{0}$ and $\dot{p}^h_{0}$ are projections of $p_{0}$ and $\dot{p}_{0}$ in $P^h$,
$w^h$ is the finite element solution to the joint problem
\begin{equation}\label{Discrete_w}
\left\{\begin{array}{ll}
-\triangle w(x,t) + c w(x,t)  = u^h_{im}(p^{\delta,h}(x,t)), & x\in \Omega_0,~t\in(t_0,\infty), \\
\frac{\partial w(x,t)}{\partial \mathbf{n}} + i w(x,t) = 0, & x\in \Gamma,~t\in(t_0,\infty),
\end{array}\right.
\end{equation}
and $u^h_{im}(p^{\delta,h}(x,t))$ is the imaginary part of the solution of (\ref{discreteBVP}), with $p^\delta$ replaced by $p^{\delta,h}$.

\begin{proposition}\label{wh}
Let $w^\delta\in \mathbf{H}^1(\Omega)$ be the weak solution of
(\ref{weakBVP}) with $p^\delta(\cdot,t)$ replaced by $p^{\delta,h}(\cdot,t)$, and $w^{\delta,h}\in \mathbf{\Psi}^h$ be the finite element solution of (\ref{Discrete_w}).
Then, a constant $C(\Omega)$ exists such that for any $p^{\delta,h}(\cdot,t) \in L^2((t_0,\infty),P^h)$, and almost every $t\in [t_0,\infty)$,
\begin{eqnarray*}
|\|w^{\delta,h}(p^{\delta,h}(\cdot,t))-w^\delta(p^{\delta,h}(\cdot,t))\||_{1,\Omega} \leq C(\Omega) h \left( \|p^{\delta,h}(\cdot,t)\|_{0,\Omega_0}
+ \|g^\delta_1\|_{0,\Gamma} +  \|g^\delta_2\|_{0,\Gamma} \right).
\label{inequality3_1}
\end{eqnarray*}
\end{proposition}

Combining Theorems \ref{ThmDynamic} and \ref{FiniteElement_u}, Proposition \ref{wh}, as well as the definition of $\Pi^h$, it is not difficult to obtain the following estimate.
\begin{proposition}\label{Convph}
Let $p^{\delta}(\cdot,t)\in P$ and $p^{\delta,h}(\cdot,t)\in P^h$ be solutions of
(\ref{DDBNandDDS}) and (\ref{DDBNandDDS_semiDiscretized}) respectively.
Then, a constant $C(\Omega)$ exists such that for almost every $t\in [t_0,\infty)$,
\begin{eqnarray*}
\|p^{\delta,h}(\cdot,t)-p^{\delta}(\cdot,t)\|_P \leq C(\Omega) h \left( \|g^\delta_1\|_{0,\Gamma} +  \|g^\delta_2\|_{0,\Gamma} \right).
\end{eqnarray*}
\end{proposition}

Now, we present the main result in this subsection.
\begin{theorem}\label{ConvFE}
(convergence of the finite element solution) Let $p^{\delta,h}\in P^h$ be solution of (\ref{DDBNandDDS_semiDiscretized}).
Suppose that for almost every $t>0$ and $\delta\geq0$, $p^{\delta}(\cdot,t)\in H^1(\Omega_0)$. Then, under the assumption of Theorem \ref{ThemRegu1}, we have the strong convergence, i.e., $p^{\delta,h}(\cdot,T^*(\delta))\to p^{\dagger}(\cdot)$ in $L^2(\Omega_0)$ as $\delta,h\to0$.
\end{theorem}

\begin{proof}
By the triangle inequality
\begin{eqnarray*}
&& \|p^{\delta,h}(\cdot,T^*(\delta)) - p^{\dagger}(\cdot)\|_{0,\Omega_0} \leq \|p^{\delta,h}(\cdot,T^*(\delta)) - p^{\delta}(\cdot,T^*(\delta))\|_{0,\Omega_0} + \\ && \|p^{\delta}(\cdot,T^*(\delta)) - \Pi^h p^{\delta}(\cdot,T^*(\delta))\|_{0,\Omega_0} + \|\Pi^h p^{\delta}(\cdot,T^*(\delta)) - \Pi^h p^{\dagger}(\cdot)\|_{0,\Omega_0} + \|\Pi^h p^{\dagger}(\cdot) - p^{\dagger}(\cdot)\|_{0,\Omega_0},
\end{eqnarray*}
it suffices to show the convergence of all terms in the right-hand side of the above inequality. The convergence of the first term follows from Proposition \ref{Convph}, while the second and fourth terms converge to 0 because of the inequality (\ref{projectionInequation}). Finally, the convergence of the third term follows from Theorem \ref{ThemRegu1} and the assertion (b) of Remark \ref{RemarkWeakConv}.
\end{proof}

Finally, we give a sketch of the finite element method for problems (\ref{prow}) and (\ref{prou}). For conciseness, by slightly abusing the notation, we rewrite
$p^{\delta,h}$, $\dot{p}^{\delta,h}$ and $\ddot{p}^{\delta,h}$ to $p^{h}$, $\dot{p}^{h}$
and $\ddot{p}^{h}$. Let $m$ be the number of the nodes of triangulation
$\mathcal{T}$, and $\{\psi_l\}^m_{l=1}$ be the nodal basis functions of the linear finite
element space $\Psi^{h}$ associated with the grid points $\{x_l\}^m_{l=1}$. Then
$u^{h}(x,t) = \sum^m_{l=1} u_l(t) \psi_l(x)$ with $u_l(t)=u^{h}(x_l,t)\in L^2((t_0,\infty),\mathbb{C})$ and
$w^{h}(x,t) = \sum^m_{l=1} w_l(t) \psi_l(x)$ with $w_l(t)=w^{h}(x_l,t)\in L^2((t_0,\infty),\mathbb{C})$.
Denote $\{x_{k_l}\}^{m_0}_{l=1}=\{x_l\}^m_{l=1}\cap \overline{\Omega}_0$, $p^{h}(x,t) = \sum^{m_0}_{l=1} p_l(t) \psi_{k_l}(x)$ with $p_l(t)=p^{h}(x_{k_l},t)\in L^2((t_0,\infty),\mathbb{R})$.
As a result, the problem (\ref{discreteBVP}) reduces to the following algebraic system with any fixed $t$:
\begin{equation}\label{system1}
\left\{\begin{array}{l}
(D+E) \mathbf{u}_{re}(t) - F \mathbf{u}_{im}(t) = B \mathbf{p}(t) + \mathbf{b}_2, \\
F \mathbf{u}_{re}(t) + (D+E) \mathbf{u}_{im}(t) = \mathbf{b}_1,
\end{array}\right.
\end{equation}
where
\begin{eqnarray*}
D=[d_{ls}]_{m\times m},  d_{ls}=\int_\Omega \nabla \psi_s \cdot \nabla \psi_l dx,\quad
E = [e_{ls}]_{m\times m},  e_{ls}=\int_\Omega  \psi_s \psi_l dx, \\
F=[f_{ls}]_{m\times m},  f_{ls}=\int_\Gamma  \psi_s \psi_l ds,
\quad B=[b_{lj}]_{m\times {m_0}}, b_{lj}=\int_{\Omega_0} \psi_l(x) \psi_{k_j}(y) dx, \\
\mathbf{b}_1=[b_{1,l}]_{m\times 1},  b_{1,l}=\int_\Gamma g^\delta_1 \psi_l ds,
\quad \mathbf{b}_2=[b_{2,l}]_{m\times 1},  b_{2,l}=\int_\Gamma g^\delta_2 \psi_l ds, \\
\mathbf{u}_{re}=[u_{re,l}]_{m\times 1},~\mathbf{u}_{im}=[u_{im,l}]_{m\times 1},
~\mathbf{p}=[p_j]_{m_0\times 1},~l,s=\overline{1,m},j=\overline{1,m_0}.
\end{eqnarray*}

Similarly, for any fixed $t$, finding a weak solution of (\ref{Discrete_w}) reduces to
solve the following system of linear equations
\begin{equation}\label{system2}
\left\{\begin{array}{l}
(D+E) \mathbf{w}_{re}(t) - F \mathbf{w}_{im}(t) = E \mathbf{u}_{im}(t), \\
F \mathbf{w}_{re}(t) + (D+E) \mathbf{w}_{im}(t) = \mathbf{0}.
\end{array}\right.
\end{equation}

\subsection{Time discretization and a novel iterative regularization algorithm}
The second order evolution equation (\ref{DDBNandDDS}) with an appropriate numerical discretization scheme for the artificial time variable yields a concrete second order iterative regularization method. The damped symplectic integrators are extremely attractive for solving second order systems, since the schemes are closely related to the canonical transformations (\cite{Hairer-2006}), and the trajectories of the discretized second flow usually kept some intrinsic invariants of the system. In this paper, we use the St\"{o}rmer-Verlet method, which belongs to the family of symplectic integrators.

Denote $q^h(x,t)=\dot{p}^h(x,t)$, and rewrite
(\ref{DDBNandDDS_semiDiscretized}) into the first order system
\begin{equation}\label{DDBNandDDS_semiDiscretized_firstOrder}
\left\{\begin{array}{ll}
\dot{q}^h = - \eta q^h - w^h_{im} \chi_{\Omega_0} , \\
\dot{p}^h = q^h, \\
p^h(t_0)=p^h_{0}, q^h(t_0)=\dot{p}^h_{0}.
\end{array}\right.
\end{equation}

Apply the St\"{o}rmer-Verlet method to the system (\ref{DDBNandDDS_semiDiscretized_firstOrder}) to obtain that at the $k$-th iteration
\begin{equation}\label{symplectic}
\left\{\begin{array}{l}
q^{h}_{k+\frac{1}{2}} = q^{h}_{k} - \frac{\Delta t}{2} \left( \eta_k q^{h}_{k} + w^{h}_{im}(p^{h}_{k}) \chi_{\Omega_0} \right), \\
p^{h}_{k+1} = p^{h}_{k} + \Delta t q^{h}_{k+\frac{1}{2}}, \\
q^{h}_{k+1} = q^{h}_{k+\frac{1}{2}} - \frac{\Delta t}{2} \left( \eta_{k+1} q^{h}_{k+\frac{1}{2}} + w^{h}_{im}(p^{h}_{k+1}) \chi_{\Omega_0} \right), \\
q^{h}(t_0)=\dot{p}^{h}_{0}, p^{h}(t_0)=p^{h}_{0},
\end{array}\right.
\end{equation}
where $p^{h}_{k}=p^{\delta,h}(t_k)$, and $\Delta t$ is the time step size.

Taking into account of the discrepancy principle for choosing the terminating time point, the newly developed numerical algorithm is proposed as follows:

\begin{algorithm}[H]
\label{Alorithm1}
\caption{The St\"{o}rmer-Verlet based SOAR for inverse source problem (\ref{BVP}).}
\begin{algorithmic}[1]
\Require Boundary data $\{g^\delta_1,g^\delta_2\}$. Noise level $\delta$. Damping parameter $\eta(t)$.
Time step size $\Delta t$. The permissible region $\Omega_0$. Triangulation $\mathcal{T}$ of domain
$\Omega$ with the nodal basis functions $\{\psi_i\}^m_{i=1}$. Precision number $\epsilon_0$.
Initial values: $(\mathbf{p}^0,\mathbf{q}^0)$. Iteration index: $k\gets0$.

\Ensure The estimated source term: $\hat{p}^{h} = \sum^{m_0}_{l=1} \mathbf{p}^{k}_l \psi_{k_l}$.

\While{$\chi(t_k)>\epsilon_0$ or $\chi_{TE}(t_k)>\epsilon_0$}

\State Solve (\ref{system1}) and (\ref{system2}) with source $\mathbf{p}^k$ to get $\mathbf{w}^k_{im}$.

\State $\mathbf{q}^{k+\frac{1}{2}} \gets
\mathbf{q}^{k} - \frac{\Delta t}{2} \left( \eta_k \mathbf{q}^{k} + \mathbf{w}^k_{im} \right)$

\State $\mathbf{p}^{k+1} \gets \mathbf{p}^{k} + \Delta t \mathbf{q}^{k+\frac{1}{2}}$

\State Solve (\ref{system1}) and (\ref{system2}) with source $\mathbf{p}^{k+1}$ to get $\mathbf{w}^{k+1}_{im}$.

\State $\mathbf{q}^{k+1} \gets \mathbf{q}^{k+\frac{1}{2}} - \frac{\Delta t}{2} \left( \eta_{k+1} \mathbf{q}^{k+\frac{1}{2}} + \mathbf{w}_{im}^{k+1}) \right)$

\State $t_{k+1} \gets t_k + \Delta t$

\State $k\gets k+1$
\EndWhile
\end{algorithmic}
\end{algorithm}


\section{Simulations}

In this section, we present some numerical examples to demonstrate the effectiveness of the proposed second order asymptotical regularization (SOAR) methods. With the problem domain $\Omega$,
Neumann data $g_2$, and a prescribed true source function $p^\dagger$ in $\Omega_0\subset \Omega$,
by using the standard linear finite element method defined in Subsection 5.1,
we solve the forward BVP
\begin{equation}
\begin{array}{ll}
 -\triangle u + u =p^\dagger\chi_{\Omega_0}
\textmd{~in~} \Omega,  \textmd{~and~} \frac{\partial u}{\partial \mathbf{n}} = g_2 \textmd{~on~}\ \Gamma
\end{array}
\label{nm1}
\end{equation}
to get $u^{h}\in \Psi^{h}$. Use $g_1=u^{h}|_{\Gamma}$ for the boundary measurement. Uniformly distributed
noises with the relative error level $\delta'$ are added to both $g_1$ and $g_2$ to get $g^{\delta}_1$ and $g^{\delta}_2$:
\[ g^{\delta}_j(x)=[1+\delta'\cdot(2\,\textrm{rand}(x)-1)]\,g_j(x),\quad x\in \Gamma,\quad j = 1, 2,\]
where rand$(x)$ returns a pseudo-random value drawn from a uniform distribution on $[0, 1]$. The noise level of measurement data is calculated by $\delta=\max_{j=1,2} \|g^{\delta}_j-g_j\|_{\infty,\Gamma}$. Then, with the noisy data $g^{\delta}_1$ and $g^{\delta}_2$, properly chosen parameters, e.g. $\eta$ and $\Delta t$,
\textbf{Algorithm 1} is implemented to get $p^{h}$~-- a stable approximation of $p^\dagger$ by SOAR.
In all experiments below, we set $g_2\equiv 0 $ on $\Gamma$, $t_0=1$ and the precision parameter $\epsilon_0=10^{-6}$.
We use $N_{max}$ as the maximal number of iterations where \textbf{Algorithm 1} stops, which may have different values in different experiments.

We refer to SOAR1 as \textbf{Algorithm 1} when $\eta$ is constant and $\chi$ is used;
SOAR2 when $\eta$ is constant and $\chi_{TE}$ is used; SOAR3 when $\eta=r/t$ and $\chi$ is used;
SOAR4 when $\eta=r/t$ and $\chi_{TE}$ is used. To assess the accuracy of the approximate
solutions, we define the $L^2$-norm relative error for an approximate solution $p^{h}$:
$ \textrm{L2Err}:= \|p^{h}-p^\dagger\|_{0,\Omega} / \|p^\dagger\|_{0,\Omega} $.
All experiments in Subsection \ref{subsec:conv}--\ref{subsec:comparison} are
implemented for the following two examples:

\smallskip

\textbf{Example 1}: $\Omega:=\{(x_1,x_2)\in\mathbb{R}^2|\,x_1^2+x^2_2<1\}$,
$\Omega_0:=\{(x_1,x_2)\in\mathbb{R}^2|\,-0.5<x_1,x_2<0.5\}$. $p^\dagger(x_1,x_2)=(1+x_1+x_2)\chi_{\Omega_0}$.  The Dirichlet data $g_1$ is computed on a
mesh with mesh size $h=0.01386$, 144929 nodes and 288768 elements.

\smallskip

\textbf{Example 2}: $\Omega$ is the same as Example 1.
$\Omega_0=\Omega_1\bigcup\Omega_2$ with
$\Omega_1:=\{(x_1,x_2)\in\mathbb{R}^2|\,(x_1+0.5)^2+x^2_2<0.01\}$
and $\Omega_2:=\{(x_1,x_2)\in\mathbb{R}^2|\,(x_1-0.5)^2+x^2_2<0.01\}$. $p^\dagger(x_1,x_2)=(1+x_1+x_2)\chi_{\Omega_1}+e^{1+x_1+x_2}\chi_{\Omega_2}$. The Dirichlet data $g_1$ is computed on a mesh with $h=0.01228$, 156225 nodes and 311296 elements.

For Example 1, all approximate sources are
reconstructed over a mesh with mesh size $h=0.1293$, 599 nodes and 1128 elements.
For Example 2, all approximate sources are
reconstructed over a mesh with mesh size $h=0.1222$, 645 nodes and 1216 elements.

\subsection{Regularization of the method}\label{subsec:conv}

\begin{figure}[!b]
\begin{center}
\subfigure[]{
\includegraphics[width=2.2in]{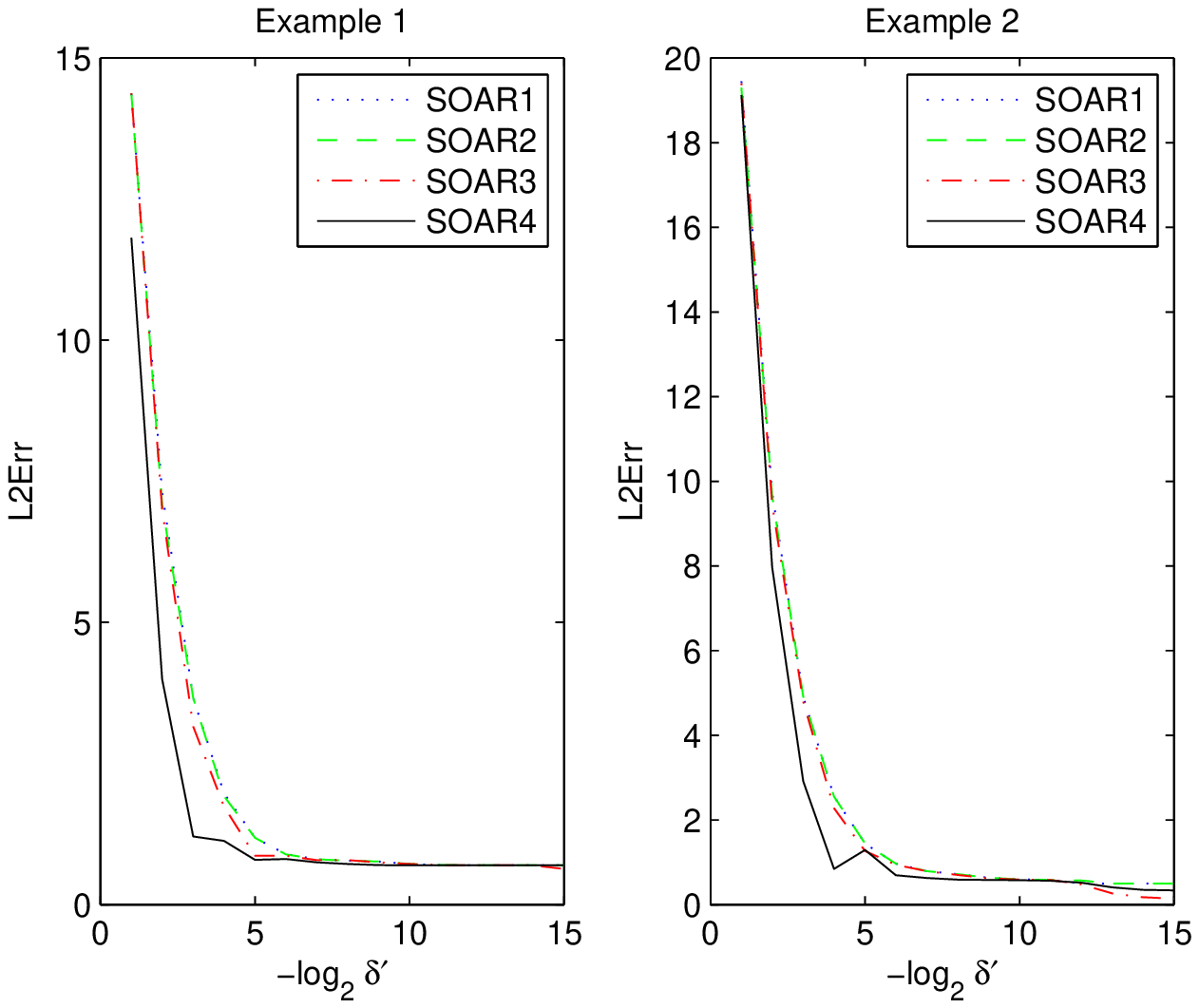}}
\subfigure[]{
\includegraphics[width=2.2in]{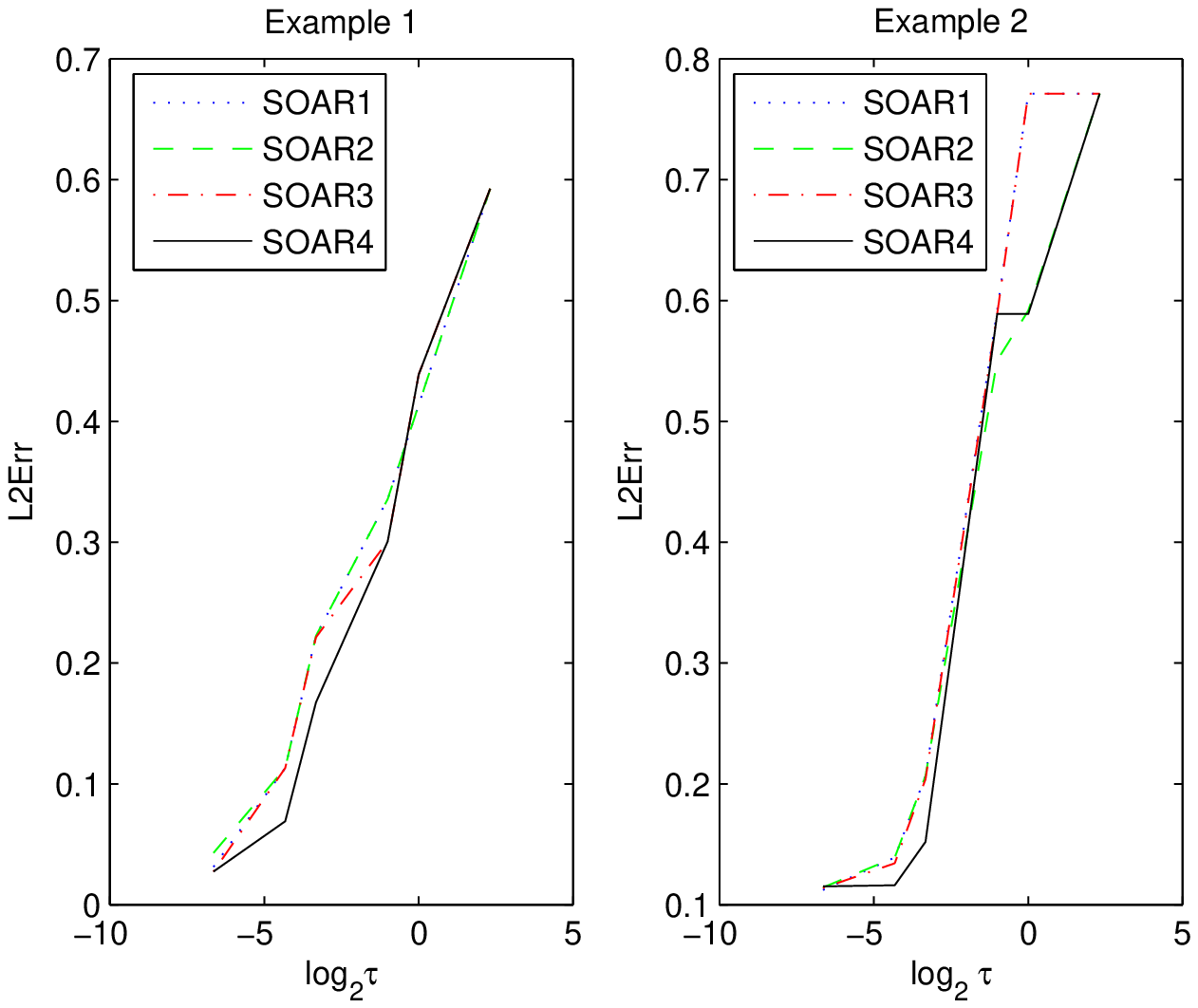}}
\end{center}
\caption{(a) Evolutions of L2Err vs. $\delta'$. (b) Evolutions of L2Err vs. $\tau$ with $\Delta t=10, \eta=0.1$ or $5/t$. }
\label{fig:L2ErrVsDelta}
\end{figure}

\vspace*{-0.1cm}

We first validate the convergence result of Theorem \ref{ThemRegu1}. \textbf{Algorithm 1} is
implemented for $\delta'=2^{-1}, 2^{-2}$, $\cdots,2^{-15}$.
As indicated by the assumptions of Lemma \ref{Rootdiscrepancy} and Theorem \ref{ThemRegu1},
let $\tau=1.1$ (used in (\ref{Morozov}) and (\ref{TotalEnergy})), $\eta=1$ when it is constant,
$\eta=5/t$ when it is dynamic, $p_0=30, q_0=\dot{p}_0=0$ in $\Omega_0$
for Example 1, and $p_0=70, q_0=\dot{p}_0=0$ in $\Omega_0$
for Example 2 so that $(p_0,\dot{p}_0)$ satisfies $\|u_{im}(p_0)\|_{0,\Omega}> C_0 \tau\delta$
and $V(p_0) + \|\dot{p}_0\|^2_P> C^2_0 \tau \delta^2$ ($C_0= 2\sqrt{2\pi}$). Moreover, for the implementation of \textbf{Algorithm 1},
set the time step $\Delta t=1$.

\begin{table}[!b]
{\footnotesize
\begin{center}
\begin{tabular}{|c|c|c|c|c|c|} \hline
\multirow{2}{*}{$\delta'$} &
\multicolumn{2}{c|}{SOAR1} &
\multicolumn{2}{c|}{SOAR2} \\
\cline{2-5}
& \textrm{L2Err} & \textrm{IterNum} & \textrm{L2Err} & \textrm{IterNum}  \\ \hline

$2^{-1}$	&14.3642	&21	&14.3642	&21	\\
$2^{-2}$	&7.3013	&43	&7.0815	&44	\\
$2^{-3}$	&3.6594	&66	&3.6594	&66	\\
$2^{-4}$	&1.9695	&88	&1.9190	&89	\\
$2^{-5}$	&1.1803	&111	&1.1803	&111\\
$2^{-6}$	&0.8950	&133	&0.8881	&134\\
$2^{-7}$	&0.8021	&156	&0.8002	&157\\
$2^{-8}$	&0.7744	&182	&0.7744	&182\\
$2^{-9}$	&0.7638	&226	&0.7638	&226\\
$2^{-10}$	&0.7171	&957	&0.7178	&936\\
$2^{-11}$	&0.7004	&1835	&0.7012	&1754\\
$2^{-12}$	&0.6961	&2723	&0.6969	&2441\\
$2^{-13}$	&0.6950	&3650	&0.6958	&2868\\
$2^{-14}$	&0.6946	&4790	&0.6955	&3043\\
$2^{-15}$	&0.6918	&$N_{max}$	&0.6955	&3095\\ \hline
\multirow{2}{*}{$\delta'$} &
\multicolumn{2}{c|}{SOAR3} &
\multicolumn{2}{c|}{SOAR4} \\
\cline{2-5}
& \textrm{L2Err} & \textrm{IterNum} & \textrm{L2Err} & \textrm{IterNum}  \\ \hline
$2^{-1}$	&14.3728	&14	&11.8109	&16	\\
$2^{-2}$	&7.0005	&20	&3.9916	&23	\\
$2^{-3}$	&3.1556	&24	&1.2075	&27	\\
$2^{-4}$	&1.7341	&26	&1.1257	&31	\\
$2^{-5}$	&0.8636	&28	&0.7910	&46	\\
$2^{-6}$	&0.8637	&28	&0.8056	&61	\\
$2^{-7}$	&0.7870	&29	&0.7453	&80	\\
$2^{-8}$	&0.7871	&29	&0.7165	&102	\\
$2^{-9}$	&0.7484	&66	&0.6993	&137	\\
$2^{-10}$	&0.7162	&102	&0.6948	&175	\\
$2^{-11}$	&0.6989	&138	&0.6949	&228	\\
$2^{-12}$	&0.6958	&156	&0.6945	&283 \\
$2^{-13}$	&0.6948	&174	&0.6944	&321	\\
$2^{-14}$	&0.6945	&282	&0.6943	&339	\\
$2^{-15}$	&0.6302	&3829	&0.6943	&350 \\\hline
\end{tabular}
\end{center}
}
\caption{Example 1: L2Err and IterNum vs $\delta'$ with $\tau=1.1, \Delta t=1, \eta=1$ or $5/t$, $N_{max}=50000$.}
\label{tab:L2ErrVsDelta}
\end{table}

The evolutions of L2-norm relative errors in approximate solutions computed
from \textbf{Algorithm 1} are plotted in (a) of Figure \ref{fig:L2ErrVsDelta},
which indicates that \textbf{Algorithm 1} for all four cases are convergent
and, thus confirms the theoretical analysis. The detailed errors and the corresponding
iterative numbers are given in Tables \ref{tab:L2ErrVsDelta} and \ref{tab:L2ErrVsDelta2},
where we can see that for both examples,
using a dynamic damping parameter $\eta(t)$ and the total energy discrepancy functional $\chi_{TE}$
can accelerate the iteration, and this is particularly remarkable when the noise level $\delta'$
is relatively small. However, as shown in Figure \ref{fig:L2ErrVsDelta}(a) and
Tables \ref{tab:L2ErrVsDelta} and \ref{tab:L2ErrVsDelta2}, compared with the noise level $\delta$, the accuracy
of the obtained approximate solution is not highly qualified. This is because the
iterations stop before getting satisfactory approximate solutions. As mentioned in
Remark \ref{Remarktau}, constants $\eta\geq 1$ and $\tau>1$ are just the sufficient
conditions for Lemmas \ref{LemmaVanishingExact} and \ref{Rootdiscrepancy}. As we
shall see in the next subsection, using smaller values of the parameters $\eta$ and $\tau$
will significantly improve the solution accuracy.

\begin{table}[!tbh]
{\footnotesize
\begin{center}
\begin{tabular}{|c|c|c|c|c|c|} \hline
\multirow{2}{*}{$\delta'$} &
\multicolumn{2}{c|}{SOAR1} &
\multicolumn{2}{c|}{SOAR2} \\
\cline{2-5}
& \textrm{L2Err} & \textrm{IterNum} & \textrm{L2Err} & \textrm{IterNum}  \\ \hline

$2^{-1}$	&19.4452	&9	&19.2942	&10	\\
$2^{-2}$	&9.7309	&98	&9.6558	&99	\\
$2^{-3}$	&4.9134	&187	&4.9134	&187	\\
$2^{-4}$	&2.5508	&276	&2.5508	&276	\\
$2^{-5}$	&1.4386	&365	&1.4386	&365 \\
$2^{-6}$	&0.9641	&456	&0.9611	&457 \\
$2^{-7}$	&0.7870	&554	&0.7870	&554 \\
$2^{-8}$	&0.7181	&690	&0.7181	&690 \\
$2^{-9}$	&0.6344	&1423	&0.6349	&1416 \\
$2^{-10}$	&0.5912	&2615	&0.5917	&2586 \\
$2^{-11}$	&0.5786	&4085	&0.5793	&3915 \\
$2^{-12}$	&0.4981	&$N_{max}$	&0.5656	&10262 \\
$2^{-13}$	&0.4981	&$N_{max}$	&0.4981	&$N_{max}$ \\
$2^{-14}$	&0.4981	&$N_{max}$	&0.4981	&$N_{max}$ \\
$2^{-15}$	&0.4981	&$N_{max}$	&0.4981	&$N_{max}$ \\ \hline
\multirow{2}{*}{$\delta'$} &
\multicolumn{2}{c|}{SOAR3} &
\multicolumn{2}{c|}{SOAR4} \\
\cline{2-5}
& \textrm{L2Err} & \textrm{IterNum} & \textrm{L2Err} & \textrm{IterNum}  \\ \hline
$2^{-1}$	&19.3905	&9	&19.1197	&10	\\
$2^{-2}$	&9.4475	&32	&7.9659	&35	\\
$2^{-3}$	&4.7961	&42	&2.9125	&47	\\
$2^{-4}$	&2.2760	&49	&0.8393	&56	\\
$2^{-5}$	&1.2621	&53	&1.2887	&67	\\
$2^{-6}$	&0.9376	&55	&0.6924	&95	\\
$2^{-7}$	&0.7884	&57	&0.6254	&129	\\
$2^{-8}$	&0.6978	&93	&0.5864	&167	\\
$2^{-9}$	&0.6223	&130	&0.5804	&222	\\
$2^{-10}$	&0.5869	&165	&0.5714	&300	\\
$2^{-11}$	&0.5792	&201	&0.5585	&409	\\
$2^{-12}$	&0.4749	&873	&0.5187	&658 \\
$2^{-13}$	&0.2628	&1714	&0.4033	&1171	\\
$2^{-14}$	&0.1703	&2165	&0.3491	&1379	\\
$2^{-15}$	&0.1372	&2451	&0.3343	&1435	\\\hline
\end{tabular}
\end{center}
}
\caption{Example 2: L2Err and IterNum vs $\delta'$ with $\tau=1.1, \Delta t=1, \eta=1$ or $5/t$, $N_{max}=50000$.}
\label{tab:L2ErrVsDelta2}
\end{table}

\subsection{Influence of parameters}\label{subsec:parameter}

The purpose of this subsection is to explore the dependence of the solution accuracy
and the convergence speed on $\tau>0$, time step size $\Delta t$,
damping parameter $\eta$ when it is constant or $r$ when $\eta(t)=r/t$, and
thus to give a guide on the choices of them in practice. For focusing on
the effect of these parameters on \textbf{Algorithm 1}, we fix $\delta'= 5\%$
in this subsection. Moreover, in the remaining part of this section,
we simply set $p_0=q_0=0$. In addition, because the parameter $\tau$
does not involve the computation of the approximate solutions itself and only affects
the iterative number where \textbf{Algorithm 1} stops,
in the following, by slightly abusing the notation, we refer $\tau$ as $C_0 \tau$.

We first investigate the influence of parameter $\tau$ on the convergence rate.
For this purpose, we additionally set $\Delta t=10$, $\eta=0.1$ when $\eta$
is constant or $\eta = 5/t$ when $\eta$ is dynamic. The detailed L2-norm
relative errors `L2Err' and the corresponding iterative numbers `IterNum' for
 different values of $\tau$ are shown in Tables \ref{tab:L2ErrVsTau} and \ref{tab:L2ErrVsTau2},
which show that on one hand, the smaller $\tau$ is, the better the solution accuracy is;
on the other hand, the smaller $\tau$ is, the more the iterative number
for stopping \textbf{Algorithm 1} is. It is no surprise that the parameter $\tau$
does not involve the computation of the approximate solutions itself. It is used in
stop criterion and only affects the iterative number where \textbf{Algorithm 1} stops.
Therefore, it is natural that a larger iterative number produces a better approximate solution,
and this also confirms the asymptotical behavior of the proposed method. The evolutions
of L2Err vs. $\tau$ for both examples and four cases of \textbf{Algorithm 1} are plotted in (b) of
Figure \ref{fig:L2ErrVsDelta}. Generally, $\tau<1$ is enough to produce reasonable approximate solutions.
Note that, as shown in Subsection \ref{subsec:conv}, bigger $\tau$ may produce
satisfactory approximate solutions when the noise level $\delta$ is rather small.

\begin{table}[H]
{\footnotesize
\begin{center}
\begin{tabular}{|c|c|c|c|c|c|} \hline
\multirow{2}{*}{$\tau$} &
\multicolumn{2}{c|}{SOAR1} &
\multicolumn{2}{c|}{SOAR2} \\
\cline{2-5}
& \textrm{L2Err} & \textrm{IterNum} & \textrm{L2Err} & \textrm{IterNum}  \\ \hline
0.01 & 0.0312 & 35 & 0.0429 & 28  \\
0.05 & 0.1131 & 15 & 0.1131 & 15   \\
0.1 & 0.2223 & 7 & 0.2223 & 7  \\
0.5 & 0.3355 & 3 & 0.3355 & 3  \\
1 & 0.4134 & 2 & 0.4134 & 2   \\
5 & 0.5925 & 1 & 0.5925 & 1   \\ \hline
\multirow{2}{*}{$\tau$} &
\multicolumn{2}{c|}{SOAR3} &
\multicolumn{2}{c|}{SOAR4} \\
\cline{2-5}
& \textrm{L2Err} & \textrm{IterNum} & \textrm{L2Err} & \textrm{IterNum}  \\ \hline
0.01 & 0.0274 & 19 & 0.0274 & 19 \\
0.05 & 0.1131 & 12 & 0.0689 & 14 \\
0.1 & 0.2212 & 8 & 0.1673 & 10 \\
0.5 & 0.3006 & 6 & 0.3006 & 6 \\
1 & 0.4388 & 4 & 0.4388 & 4  \\
5 & 0.5925 & 1 & 0.5925 & 1 \\ \hline
\end{tabular}
\end{center}
}
\caption{Example 1: L2Err and IterNum vs $\tau$ with $\Delta t=10, \eta=0.1$ or $5/t$. }
\label{tab:L2ErrVsTau}
\end{table}

\begin{table}[H]
{\footnotesize
\begin{center}
\begin{tabular}{|c|c|c|c|c|c|} \hline
\multirow{2}{*}{$\tau$} &
\multicolumn{2}{c|}{SOAR1} &
\multicolumn{2}{c|}{SOAR2} \\
\cline{2-5}
& \textrm{L2Err} & \textrm{IterNum} & \textrm{L2Err} & \textrm{IterNum}  \\ \hline
0.01 & 0.1123 & 58 & 0.1143 & 49  \\
0.05 & 0.1391 & 31 & 0.1391 & 31   \\
0.1 & 0.2065 & 20 & 0.2065 & 20  \\
0.5 & 0.5914 & 2 & 0.5504 & 3  \\
1 & 0.7709 & 1 & 0.5914 & 2   \\
5 & 0.7709 & 1 & 0.7709 & 1   \\ \hline
\multirow{2}{*}{$\tau$} &
\multicolumn{2}{c|}{SOAR3} &
\multicolumn{2}{c|}{SOAR4} \\
\cline{2-5}
& \textrm{L2Err} & \textrm{IterNum} & \textrm{L2Err} & \textrm{IterNum}  \\ \hline
0.01 & 0.1137 & 20 & 0.1150 & 29 \\
0.05 & 0.1342 & 17 & 0.1159 & 19 \\
0.1 & 0.2037 & 14 & 0.1519 & 16 \\
0.5 & 0.5887 & 2 & 0.5887 & 2 \\
1 & 0.7709 & 1 & 0.5887 & 2  \\
5 & 0.7709 & 1 & 0.7709 & 1 \\ \hline
\end{tabular}
\end{center}
}
\caption{Example 2: L2Err and IterNum vs $\tau$ with $\Delta t=10, \eta=0.1$ or $5/t$. }
\label{tab:L2ErrVsTau2}
\end{table}

Now we investigate the influence of time step size $\Delta t$ on the
solution accuracy and the convergence rate. To this end, set $\tau=0.01$,
$\eta = 0.1$ or $5/t$. The L2-norm relative errors 'L2Err'
and the corresponding iterative numbers 'IterNum' for both examples and four algorithms
are given in Tables \ref{tab:L2ErrVsDt} and \ref{tab:L2ErrVsDt2}, which show that the
bigger the time step size $\Delta t$ is, the faster the iteration is.
However, our experiments suggest that $\Delta t$ should not be too big.
Otherwise, the iteration will blow up as it breaks the consistency of the numerical scheme.  The evolutions of L2Err vs. $\Delta t$ are plotted in Figure \ref{fig:L2ErrVsDt}. In the remaining experiments, we choose $\Delta t=10$.

\begin{table}[H]
{\footnotesize
\begin{center}
\begin{tabular}{|c|c|c|c|c|c|} \hline
\multirow{2}{*}{$\Delta t$} &
\multicolumn{2}{c|}{SOAR1} &
\multicolumn{2}{c|}{SOAR2} \\
\cline{2-5}
& \textrm{L2Err} & \textrm{IterNum} & \textrm{L2Err} & \textrm{IterNum}  \\ \hline
0.01 & 0.3859 & $N_{max}$ & 0.3859 & $N_{max}$  \\
0.05 & 0.2709 & $N_{max}$ & 0.2709 & $N_{max}$  \\
0.1 & 0.1758 & $N_{max}$ & 0.1758 & $N_{max}$  \\
0.5 & 0.0322 & 677 & 0.0432 & 556  \\
1 & 0.0322 & 339 & 0.0433 & 278  \\
5 & 0.0317 & 69 & 0.0430 & 56  \\
10 & 0.0312 & 35 & 0.0429 & 28  \\ \hline
\multirow{2}{*}{$\Delta t$} &
\multicolumn{2}{c|}{SOAR3} &
\multicolumn{2}{c|}{SOAR4} \\
\cline{2-5}
& \textrm{L2Err} & \textrm{IterNum} & \textrm{L2Err} & \textrm{IterNum}  \\ \hline
0.01 & 0.7744 & $N_{max}$ & 0.7744 & $N_{max}$  \\
0.05 & 0.3478 & $N_{max}$ & 0.3178 & $N_{max}$  \\
0.1 & 0.1800 & $N_{max}$ & 0.1800 & $N_{max}$  \\
0.5 & 0.0313 & 332 & 0.0260 & 363  \\
1 & 0.0313 & 166 & 0.0261 & 182  \\
5 & 0.0284 & 34 & 0.0258 & 36  \\
10 & 0.0274 & 19 & 0.0274 & 19  \\ \hline
\end{tabular}
\end{center}
}
\caption{Example 1: L2Err and IterNum vs $\Delta t$ with $\tau=0.01, \eta=0.1$ or $5/t$, $N_{max}=1000$. }
\label{tab:L2ErrVsDt}
\end{table}

\begin{table}[H]
{\footnotesize
\begin{center}
\begin{tabular}{|c|c|c|c|c|c|} \hline
\multirow{2}{*}{$\Delta t$} &
\multicolumn{2}{c|}{SOAR1} &
\multicolumn{2}{c|}{SOAR2} \\
\cline{2-5}
& \textrm{L2Err} & \textrm{IterNum} & \textrm{L2Err} & \textrm{IterNum}  \\ \hline
0.01 & 0.8353 & $N_{max}$ & 0.8353 & $N_{max}$  \\
0.05 & 0.5027 & $N_{max}$ & 0.5027 & $N_{max}$  \\
0.1 & 0.3616 & $N_{max}$ & 0.3616 & $N_{max}$  \\
0.5 & 0.1123 & $N_{max}$ & 0.1145 & 965  \\
1 & 0.1123 & 576 & 0.1145 & 483  \\
5 & 0.1123 & 116 & 0.1145 & 97  \\
10 & 0.1123 & 58 & 0.1143 & 49 \\ \hline
\multirow{2}{*}{$\Delta t$} &
\multicolumn{2}{c|}{SOAR3} &
\multicolumn{2}{c|}{SOAR4} \\
\cline{2-5}
& \textrm{L2Err} & \textrm{IterNum} & \textrm{L2Err} & \textrm{IterNum}  \\ \hline
0.01 & 0.9521 & $N_{max}$ & 0.9521 & $N_{max}$  \\
0.05 & 0.5570 & $N_{max}$ & 0.5570 & $N_{max}$  \\
0.1 & 0.3691 & $N_{max}$ & 0.3691 & $N_{max}$  \\
0.5 & 0.1137 & 396 & 0.1134 & 615 \\
1 & 0.1137 & 198 & 0.1134 & 307  \\
5 & 0.1138 & 40 & 0.1142 & 60  \\
10 & 0.1137 & 20 & 0.1150 & 29 \\ \hline
\end{tabular}
\end{center}
}
\caption{Example 2: L2Err and IterNum vs $\Delta t$ with $\tau=0.01, \eta=0.1$ or $5/t$, $N_{max}=1000$. }
\label{tab:L2ErrVsDt2}
\end{table}

\begin{figure}[!t]
\begin{center}
\includegraphics[width=2.5in]{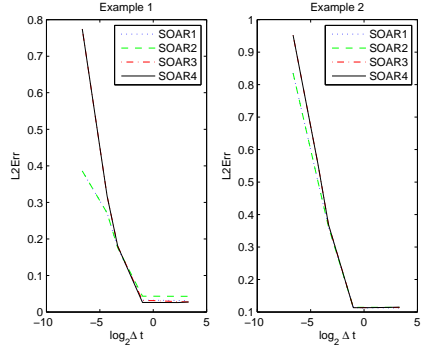}
\end{center}
\caption{Evolutions of L2Err vs. $\Delta t$  with $\tau=0.01, \eta=0.1$ or $5/t$.}
\label{fig:L2ErrVsDt}
\end{figure}

We next discuss the influence of the damping parameter $\eta$ on the
solution accuracy and the convergence rate. In the experiments,
set $\tau=0.01, \Delta t=10$. For constant $\eta$, the L2-norm relative errors 'L2Err'
and the corresponding iterative numbers 'IterNum' are given in Tables \ref{tab:L2ErrVsEta} and \ref{tab:L2ErrVsEta2} from which we conclude that $\eta\leq 0.1$ can lead to reasonable approximate solutions
for \textbf{Algorithm 1} for four cases. Nevertheless, $\eta$ should not be too small.
Too small $\eta$ brings oscillation in solution accuracy.  The evolutions of L2Err vs. $\eta$ are shown in Figure \ref{fig:L2ErrVsDt2}.  For dynamic damping parameter $\eta = r/t$, the
L2-norm relative errors 'L2Err' and the corresponding iterative numbers 'IterNum'
are given in Tables \ref{tab:L2ErrVsEta} and \ref{tab:L2ErrVsEta2}. The evolutions of 'L2Err' vs. the factor $r$
are also shown in Figure \ref{fig:L2ErrVsDt2}. Both Tables \ref{tab:L2ErrVsEta}, \ref{tab:L2ErrVsEta2} and
Figure \ref{fig:L2ErrVsDt2} indicate that, like $\eta$, the factor $r$ should
be neither too small nor too big. Too small $r$ also brings oscillation
in solution accuracy. Therefore,
in the remaining experiments, set $\eta=0.05$ when it is constant
while set $r=5$ when $\eta=r/t$.

\begin{figure}[H]
\begin{center}
\includegraphics[width=2.9in]{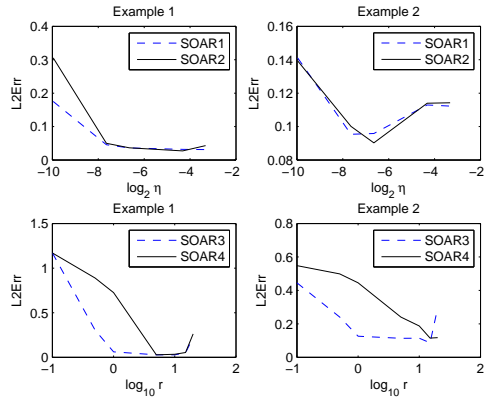}
\end{center}
\caption{Evolutions of L2Err vs. $\eta$ with $\tau=0.01, \Delta t=10$.}
\label{fig:L2ErrVsDt2}
\end{figure}

Finally, we discuss the choice of the initial data $(p_0,\dot{p}_0)$ for SOAR. According to the numerical experiments (for the concision of the statement, we omit the related numerical results), in most cases, the initial data $(p_0,\dot{p}_0)$ does not effect the result quality (the value of "L2Err"), but may influence the algorithm speed.
The closer the initial data $(p_0,\dot{p}_0)$ is to the unknown exact solution, the less of
the iteration number ``IterNum'' is required. Without knowledge of the exact solution, we recommend to set $p_0=\dot{p}_0=0$.

\subsection{Comparison with other methods}\label{subsec:comparison}

\begin{table}[!t]
{\footnotesize
\begin{center}
\begin{tabular}{|c|c|c|c|c|c|} \hline
\multirow{2}{*}{$\eta=const.$} &
\multicolumn{2}{c|}{SOAR1} &
\multicolumn{2}{c|}{SOAR2} \\
\cline{2-5}
& \textrm{L2Err} & \textrm{IterNum} & \textrm{L2Err} & \textrm{IterNum}  \\ \hline
$\eta=0.001$ & 0.1752 & 395 & 0.3057 & 849  \\
$\eta=0.005$ & 0.0455 & 97 & 0.0499 & 171  \\
$\eta=0.01$ & 0.0360 & 52 & 0.0365 & 87 \\
$\eta=0.05$ & 0.0316 & 14 & 0.0270 & 17 \\
$\eta=0.1$ & 0.0312 & 35 & 0.0429 & 28 \\ \hline
\multirow{2}{*}{$\eta=r/t$} &
\multicolumn{2}{c|}{SOAR3} &
\multicolumn{2}{c|}{SOAR4} \\
\cline{2-5}
& \textrm{L2Err} & \textrm{IterNum} & \textrm{L2Err} & \textrm{IterNum}  \\ \hline
$r=0.1$ & 1.1695 & $N_{max}$ & 1.1695 & $N_{max}$  \\
$r=0.5$ & 0.3030 & 440 & 0.8882 & $N_{max}$  \\
$r=1$ & 0.0616 & 131 & 0.7251 & $N_{max}$  \\
$r=5$ & 0.0274 & 19 & 0.0274 & 19  \\
$r=10$ & 0.0280 & 26 & 0.0343 & 23  \\
$r=15$ & 0.0503 & 40 & 0.0536 & 43  \\
$r=20$ & 0.2220 & 153 & 0.2622 & 139 \\ \hline
\end{tabular}
\end{center}
}
\caption{Example 1: L2Err and IterNum vs $\eta=const.$ or $r/t$  with $\tau=0.01, \Delta t=10$, $N_{max}=1000$. }
\label{tab:L2ErrVsEta}
\end{table}

\begin{table}[!t]
{\footnotesize
\begin{center}
\begin{tabular}{|c|c|c|c|c|c|} \hline
\multirow{2}{*}{$\eta=const.$} &
\multicolumn{2}{c|}{SOAR1} &
\multicolumn{2}{c|}{SOAR2} \\
\cline{2-5}
& \textrm{L2Err} & \textrm{IterNum} & \textrm{L2Err} & \textrm{IterNum}  \\ \hline
$\eta=0.001$ & 0.1410 & 208 & 0.1393 & 885  \\
$\eta=0.005$ & 0.0954 & 118 & 0.1001 & 176   \\
$\eta=0.01$ & 0.0958 & 49 & 0.0903 & 87 \\
$\eta=0.05$ & 0.1129 & 25 & 0.1140 & 22 \\
$\eta=0.1$ & 0.1123 & 58 & 0.1143 & 49 \\ \hline
\multirow{2}{*}{$\eta=r/t$} &
\multicolumn{2}{c|}{SOAR3} &
\multicolumn{2}{c|}{SOAR4} \\
\cline{2-5}
& \textrm{L2Err} & \textrm{IterNum} & \textrm{L2Err} & \textrm{IterNum}  \\ \hline
$r=0.1$ & 0.4453 & $N_{max}$ & 0.4453 & $N_{max}$  \\
$r=0.5$ & 0.2407 & $N_{max}$ & 0.2407 & $N_{max}$  \\
$r=1$ & 0.1261 & 136 & 0.1875 & $N_{max}$  \\
$r=5$ & 0.1137 & 20 & 0.1150 & 29  \\
$r=10$ & 0.1146 & 25 & 0.1175 & 22  \\
$r=15$ & 0.0863 & 46 & 0.0868 & 56  \\
$r=20$ & 0.2945 & 219 & 0.3201 & 195 \\ \hline
\end{tabular}
\end{center}
}
\caption{Example 2: L2Err and IterNum vs $\eta=const.$ or $r/t$  with $\tau=0.01, \Delta t=10$, $N_{max}=1000$. }
\label{tab:L2ErrVsEta2}
\end{table}

In this subsection, we compare the behaviors regarding the solution accuracy
and the convergence rate between SOAR and three existing methods; that is,
the Nesterov's method, the $\nu$-method and the dynamical regularization method (DRM)
proposed in~\cite{ZhangYe2018}. Recall that we use $\mathbf{p}$ as the coefficients of the finite element solution $p^h$, see
\textbf{Algorithm 1} for the detail. In all methods, we set $\tau=0.01$,
$\mathbf{p}^0=\mathbf{0}$, $\mathbf{q}^0=\mathbf{0}$ if $\mathbf{q}$ is involved,
and $\mathbf{p}^1=\mathbf{p}^0$ if the method is a two-step one. Moreover,
in SOAR2 and SOAR4, the total energy discrepancy priniciple $\chi_{TE}$ is used,
while, in all other methods, the usual discrepancy function $\chi$ is used.

For methods SOAR1-SOAR4, set $\Delta t=10, \eta=0.05$ or $5/t$. We remark that on the one hand,
these chosen parameters are not the optimal ones; on the other hand, a large
range of values of these parameters could produce satisfactory approximate sources $p^h$.

For the inverse source problem (\ref{BVP}) with CCBM formulation, DRM yields the following iteration
\begin{eqnarray}\label{DRM}
\left\{\begin{array}{ll}
\mathbf{q}^{k+1} =
\frac{1}{1+\eta\Delta t } \mathbf{q}^{k} - \frac{\Delta t}{1+\eta\Delta t }
\left( \mathbf{w}^k_{im} + \varepsilon(t_k) \mathbf{p}^{k} \right),  \\
\mathbf{p}^{k+1} = \mathbf{p}^{k} + \Delta t \mathbf{q}^{k+1},
\end{array}\right.
k= 0,1,\cdots,
\end{eqnarray}
where $(\mathbf{w}_{re}^k,\mathbf{w}_{im}^k)$ solves (\ref{system2}) with $\mathbf{u}_{im}$
replaced by $\mathbf{u}_{im}^k$, and $(\mathbf{u}_{re}^k,\mathbf{u}_{im}^k)$ solves
(\ref{system1}) with $\mathbf{p}$ replaced by $\mathbf{p}^k$. As suggested by numerical experiments of \cite{ZhangYe2018}, we set $\eta=1, \Delta t=10$ and the regularization parameter $\varepsilon(t)=0.1/(t\ln(t))$. It should be mentioned that DRM is not an acceleration method.

For the $\nu$-method, it is defined as~(\cite[\S~6.3]{engl1996regularization})
\begin{eqnarray}\label{nuMethod}
\begin{array}{ll}
\mathbf{p}^{k+1} = \mathbf{p}^{k} + \mu_k(\mathbf{p}^{k}-\mathbf{p}^{k-1})
- \omega_k \mathbf{w}_{im}^k, \quad k=1, 2, \cdots
\end{array}
\end{eqnarray}
with $\mu_1=0, \omega_1=(4\nu+2)/(4\nu+1)$ and
\begin{eqnarray*}
\mu_k= \frac{(k-1)(2k-3)(2k+2\nu-1)}{(k+2\nu-1)(2k+4\nu-1)(2k+2\nu-3)}, ~
\omega_k= 4 \frac{(2k+2\nu-1)(k+\nu-1)}{(k+2\nu-1)(2k+4\nu-1)}.
\end{eqnarray*}
Note that $\mathbf{w}_{im}^k$ in (\ref{nuMethod}) has the same meaning as that in (\ref{DRM}).
We select the Chebyshev method as our special $\nu$-method, i.e., $\nu=1/2$.
Moreover, set $\mathbf{p}^1=\mathbf{p}^0=0$ for the implementation of (\ref{nuMethod}).

The Nesterov's method is defined by~(\cite{Neubauer-2017})
\begin{eqnarray}\label{Nesterov}
\left\{\begin{array}{ll}
\mathbf{z}_{k} = \mathbf{p}^{k} + \frac{k-1}{k+\alpha-1} \left( \mathbf{p}^{k} - \mathbf{p}^{k-1} \right),  \\
\mathbf{p}^{k+1} = \mathbf{z}_{k} - \omega \mathbf{w}_{im}^k,
\end{array}\right.
k= 1,2,\cdots,
\end{eqnarray}
where $\alpha\geq 3$, $\mathbf{w}^k_{im}$  has the same definition as that in (\ref{DRM}) and (\ref{nuMethod}). We apply (\ref{Nesterov}) to Examples 1 and 2 with parameters $\alpha=3$ and $\omega=10$.

\begin{table}[!b]
{\footnotesize
\begin{center}
\begin{tabular}{|c|c|c|c|c|c|c|} \hline
$\delta'$ &
\multicolumn{2}{c|}{$5\%$} &
\multicolumn{2}{c|}{$10\%$} &
\multicolumn{2}{c|}{$20\%$}  \\ \hline
\multicolumn{7}{|c|}{\textbf{Example 1}} \\  \hline
\cline{2-7}
Methods & \textrm{L2Err} & \textrm{IterNum}  & \textrm{L2Err} & \textrm{IterNum}
 & \textrm{L2Err} & \textrm{IterNum} \\ \hline
DRM      & 0.0322 & 369      & 0.0571 & 314	    & 0.1260 & 219   \\
$\nu$    & 0.0164 &	53		 & 0.0491 & 51		& 0.1183 & 47   \\
Nesterov & 0.0279 &	42  	 & 0.0490 & 37	    & 0.0969 & 36    \\
SOAR1    & 0.0316 &	14		 & 0.0484 & 14		& 0.1214 & 10   \\
SOAR2    & 0.0270 & 17		 & 0.0426 & 17		& 0.0909 & 14   \\
SOAR3    & 0.0274 &	19		 & 0.0533 & 16		& 0.1079 & 15   \\
SOAR4    & 0.0274 &	19		 & 0.0420 & 18		& 0.0958 & 16   \\ \hline
\multicolumn{7}{|c|}{\textbf{Example 2}}  \\  \hline
DRM      & 0.1119 & 630	     & 0.1089 & 515		& 0.1215 & 372   \\
$\nu$    & 0.1103 & 124	     & 0.1036 & 123		& 0.1096 & 122  \\
Nesterov & 0.1095 & 87		 & 0.1114 & 44		& 0.1159 & 42    \\
SOAR1    & 0.1123 & 58		 & 0.1095 & 48		& 0.1201 & 36   \\
SOAR2    & 0.1143 & 49		 & 0.1109 & 45		& 0.1219 & 35   \\
SOAR3    & 0.1137 & 20		 & 0.1105 & 20		& 0.1169 & 18  \\
SOAR4    & 0.1137 & 29		 & 0.1152 & 23		& 0.1106 & 20   \\ \hline
\end{tabular}
\end{center}
}
\caption{Comparison with the state-of-the-art methods. }
\label{tab:comparison}
\end{table}

The results of the simulations are presented in Table \ref{tab:comparison}, from which
we conclude that, with properly chosen parameters, all the mentioned methods are stable and can produce satisfactory solutions. Compared with the dynamical regularization method, all of the other methods offer good results with similar accuracy, but require considerably fewer iterations. Particularly, SOAR1--SOAR4 converge even faster than the well-known Nesterov's method and the $\nu$-method. On the whole, for both Examples, the total energy discrepancy function $\chi_{TE}$
leads to more accurate solution than the conventional discrepancy function $\chi$, but with
slightly more iterative numbers.

We finally plot the exact and recovered sources with different methods corresponding to $\delta'=10\%$ in Figure \ref{fig:source11} for Example 1. The counterparts for Example 2
are shown in Figure \ref{fig:source21}. For the conciseness of the paper, we
omit the figures corresponding to $\delta'=5\%$ and $20\%$.

\begin{figure}[tbhp]
\begin{center}
\subfigure{
\includegraphics[width=2.5in]{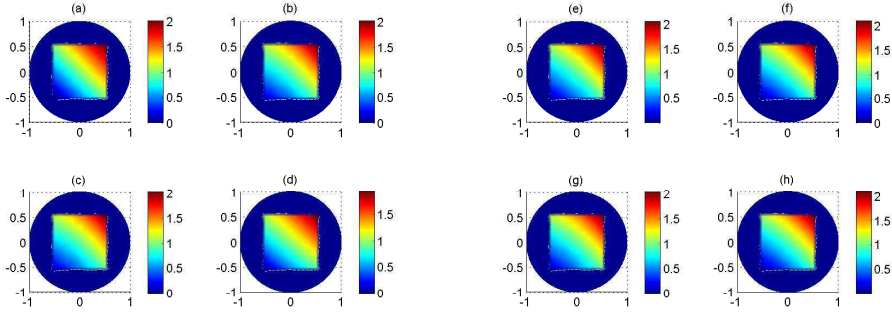}
}
\subfigure{
\includegraphics[width=2.5in]{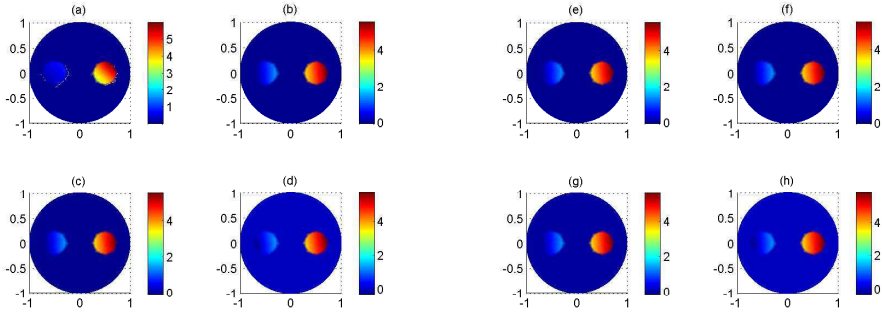}
}
\end{center}
\caption{The true and approximate sources. (a): $p^\dagger$; (b): $p^h$ by DRM;
  (c): $p^h$ by Nesterov's method (d): $p^h$ by $\nu$-method; (e): $p^h$ by SOAR1;
  (f): $p^h$ by SOAR2; (g): $p^h$ by SOAR3; (h): $p^h$ by SOAR4.}
\label{fig:source11}
\end{figure}

\begin{figure}[!htb]
\begin{center}
\subfigure{
\includegraphics[width=2.5in]{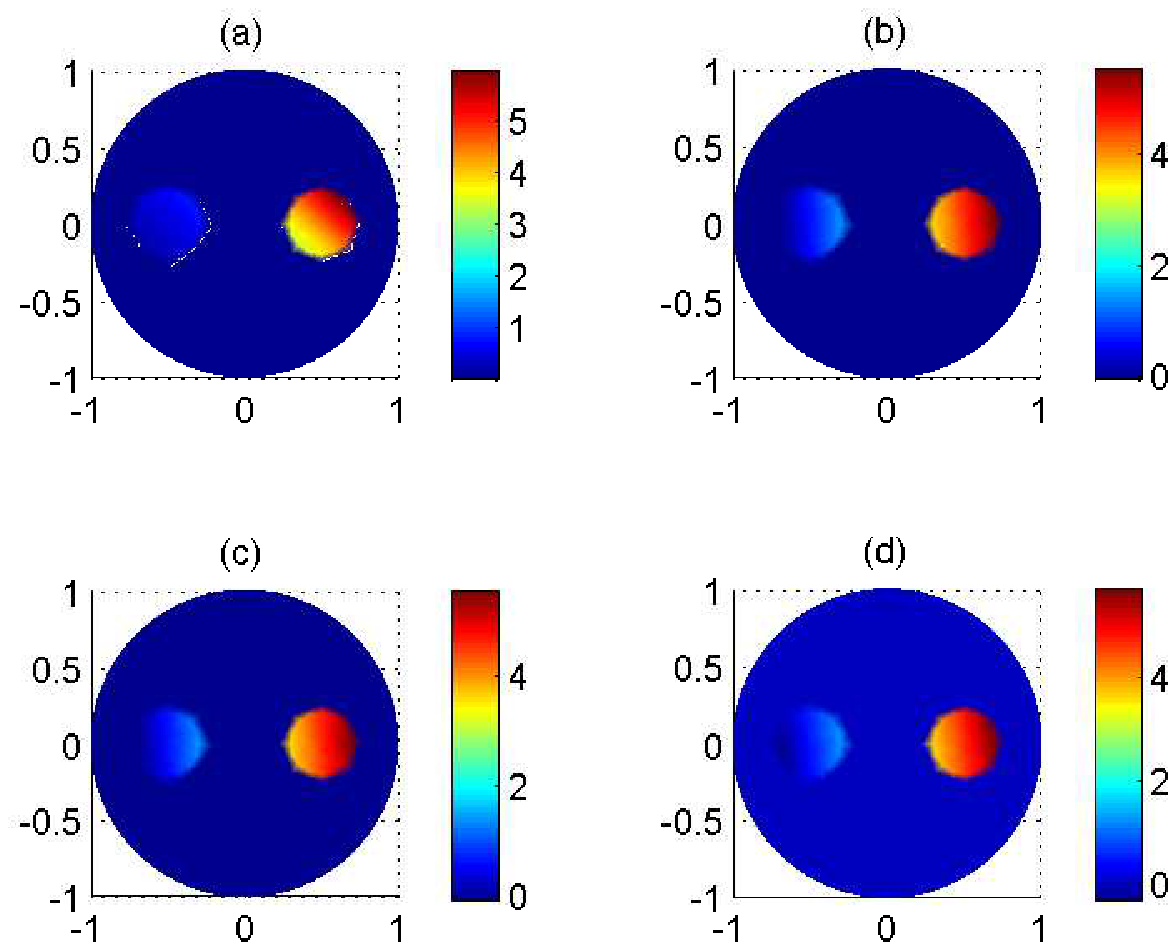}
}
\subfigure{
\includegraphics[width=2.5in]{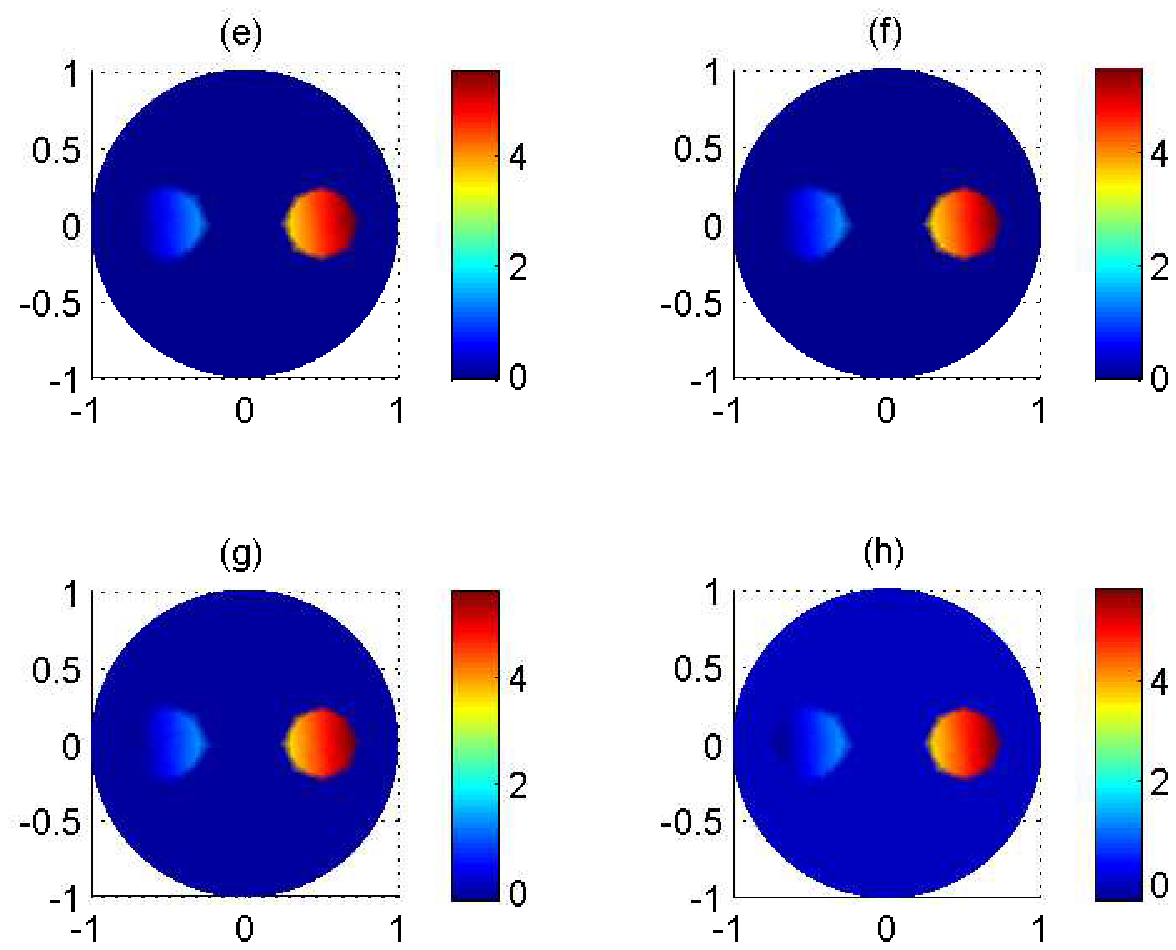}
}
\end{center}
\caption{The true and approximate sources. (a): $p^\dagger$; (b): $p^h$ by DRM;
  (c): $p^h$ by Nesterov's method (d): $p^h$ by $\nu$-method; (e): $p^h$ by SOAR1;
  (f): $p^h$ by SOAR2; (g): $p^h$ by SOAR3; (h): $p^h$ by SOAR4.}
\label{fig:source21}
\end{figure}


\section{Conclusions}

This paper is devoted to developing Second Order Asymptotical Regularization (SOAR) methods for solving inverse source problems of elliptic partial differential equations given Dirichlet
and Neumann boundary data. We show the convergence results of SOAR for both fixed and dynamic damping parameters. A symplectic scheme is applied for the numerical
implementation of SOAR. This scheme yields a novel iterative regularization method. As shown by the numerical results, the proposed SOAR methods are comparable to
the Nesterov's acceleration method and the $\nu$-method about the convergence rate.
Moreover, in this paper, a conventional Morozov's discrepancy principle and a new total
energy discrepancy principle are used for the stop criterion. Numerical experiments demonstrate that, in most cases, the newly developed total energy discrepancy principle works slightly better than the conventional Morozov's discrepancy principle. Similar to the Nesterov's acceleration method, the introduced SOAR can also be used to solve to non-linear ill-posed problems in partial differential equations, which will be the one of the topics of our future work.

\section*{Acknowledgements}

The work of Y. Zhang is supported by the Alexander von Humboldt foundation through
a postdoctoral researcher fellowship. The work of R. Gong is supported by the Natural Science Foundation of China (No. 11401304) and the Fundamental Research Funds for the Central Universities (No. NS2018047)


\bibliographystyle{IMANUM-BIB}
\bibliography{InverseSourceAccelerated}

\begin{thebibliography}{}

\bibitem[Afraites {\em et~al.}(2007)Afraites, Dambrine, \&
  Kateb]{Afraites:2007}
{\sc Afraites, L., Dambrine, M. \& Kateb, D.} (2007)
\newblock Conformal mappings and shape derivatives for the transmission problem
  with a single measurement.
\newblock {\em Numer. Func. Anal. Opt.}, {\bf 28}, 519--551.

\bibitem[Alves {\em et~al.}(2009)Alves, Martins, \& Roberty]{Alves:2009}
{\sc Alves, C., Martins, N. \& Roberty, N.} (2009)
\newblock Full identification of acoustic sources with multiple frequencies and
  boundary measurements,.
\newblock {\em Inverse Probl. Imaging\/}, {\bf 3}, 275--294.

\bibitem[Atkinson \& Han(2009)Atkinson \& Han]{Atkinson:2009}
{\sc Atkinson, K. \& Han, W.} (2009)
\newblock {\em Theoretical Numerical Analysis: A Functional Analysis Framework
  (3rd ed.)\/}.
\newblock New York: Springer-Verlag.

\bibitem[Attouch {\em et~al.}(2000)Attouch, Goudou, \& Redont]{Attouch-2000}
{\sc Attouch, H., Goudou, X. \& Redont, P.} (2000)
\newblock The heavy ball with friction method. i. the continuous dynamical
  system.
\newblock {\em Comm. Contemp. Math.}, {\bf 2}, 1--34.

\bibitem[Attouch {\em et~al.}(2018)Attouch, Chbani, Peypouquet, \&
  Redont]{Attouch2018}
{\sc Attouch, H., Chbani, Z., Peypouquet, J. \& Redont, P.} (2018)
\newblock Fast convergence of inertial dynamics and algorithms with asymptotic
  vanishing viscosity.
\newblock {\em Mathematical Programming\/}, {\bf 168}, 123--175.

\bibitem[Attouch \& Peypouquet(2016)Attouch \& Peypouquet]{Attouch-2016}
{\sc Attouch, H. \& Peypouquet, J.} (2016)
\newblock The rate of convergence of nesterov's accelerated forward-backward
  method is actually faster than $\mathcal{O}(1/k^2)$.
\newblock {\em SIAM Journal on Optimization\/}, {\bf 26}, 1824--1834.

\bibitem[Cheng {\em et~al.}(2014)Cheng, Gong, Han, \& Zheng]{Cheng:2014}
{\sc Cheng, X., Gong, R., Han, W. \& Zheng, X.} (2014)
\newblock A novel coupled complex boundary method for inverse source problems.
\newblock {\em Inverse Problems\/}, {\bf 30}, 055002.

\bibitem[Engl {\em et~al.}(1996)Engl, Hanke, \&
  Neubauer]{engl1996regularization}
{\sc Engl, H., Hanke, M. \& Neubauer, A.} (1996)
\newblock {\em Regularization of inverse problems\/},  vol. 375.
\newblock Springer.

\bibitem[Fernandez~Bonder \& Rossi(2001)Fernandez~Bonder \& Rossi]{Bonder2001}
{\sc Fernandez~Bonder, J. \& Rossi, J.} (2001)
\newblock Existence results for the $p$-laplacian with nonlinear boundary
  conditions.
\newblock {\em J. Math. Anal., Appl.}, {\bf 263}, 195--223.

\bibitem[Hairer {\em et~al.}(2006)Hairer, Wanner, \& Lubich]{Hairer-2006}
{\sc Hairer, E., Wanner, G. \& Lubich, C.} (2006)
\newblock {\em Geometric Numerical Integration: Structure-Preserving Algorithms
  for Ordinary Differential Equations (Second Edition)\/}.
\newblock New York: Springer.

\bibitem[Han {\em et~al.}(2006)Han, Cong, \& Wang]{Han:2006}
{\sc Han, W., Cong, W. \& Wang, G.} (2006)
\newblock Mathematical theory and numerical analysis of bioluminescence
  tomography.
\newblock {\em Inverse Problems\/}, {\bf 22}, 1659--1675.

\bibitem[Hanke {\em et~al.}(1995)Hanke, Neubauer, \& Scherzer]{Scherzer-1995}
{\sc Hanke, M., Neubauer, A. \& Scherzer, O.} (1995)
\newblock A convergence analysis of the landweber iteration for nonlinear
  ill-posed problems.
\newblock {\em Numerische Mathematik\/}, {\bf 72}, 21--37.

\bibitem[Hubmer \& Ramlau(2017)Hubmer \& Ramlau]{Hubmer-2017}
{\sc Hubmer, S. \& Ramlau, R.} (2017)
\newblock Convergence analysis of a two-point gradient method for nonlinear
  ill-posed problems.
\newblock {\em Inverse Problems\/}, {\bf 33}, 095004.

\bibitem[Isakov(1990)Isakov]{Isakov:1990}
{\sc Isakov, V.} (1990)
\newblock {\em Inverse Source Problems\/}.
\newblock New York: American Mathematical Society.

\bibitem[Johnson(2009)Johnson]{Johnson-2009}
{\sc Johnson, C.} (2009)
\newblock {\em Numerical Solution of Partial Differential Equations by the
  Finite Element Method\/}.
\newblock Mineola: Dover.

\bibitem[Kaltenbacher {\em et~al.}(2008)Kaltenbacher, Neubauer, \&
  Scherzer]{Kaltenbacher-2008}
{\sc Kaltenbacher, B., Neubauer, A. \& Scherzer, O.} (2008)
\newblock {\em Iterative regularization methods for nonlinear ill-posed
  problems\/}.
\newblock Berlin: Walter de Gruyter GmbH \& Co. KG.

\bibitem[Motron(2002)Motron]{Motron-2002}
{\sc Motron, M.} (2002)
\newblock Around the best constants for the sobolev trace map from
  $w^{1,2}(\omega)$ into $l^{1}(\partial\omega)$.
\newblock {\em Asymptotic Analysis\/}, {\bf 29}, 69--90.

\bibitem[Neubauer(2017)Neubauer]{Neubauer-2017}
{\sc Neubauer, A.} (2017)
\newblock On nesterov acceleration for landweber iteration of linear ill-posed
  problems.
\newblock {\em Journal of Inverse and Ill-Posed Problems\/}, {\bf 25},
  381--390.

\bibitem[Opial(1967)Opial]{Opial}
{\sc Opial, Z.} (1967)
\newblock Weak convergence of the sequence of successive approximations for
  nonexpansive mappings.
\newblock {\em Bull. of the Amer. Math. Soc.}, {\bf 73}, 591--597.

\bibitem[Song \& Huang(2012)Song \& Huang]{Song-2012}
{\sc Song, S. \& Huang, J.} (2012)
\newblock Solving an inverse problem from bioluminescence tomography by
  minimizing an energy-like functional.
\newblock {\em J. Comput. Anal. Appl.}, {\bf 14}, 544--558.

\bibitem[Su {\em et~al.}(2016)Su, Boyd, \& Candes]{Su-2016}
{\sc Su, W., Boyd, S. \& Candes, E.} (2016)
\newblock A differential equation for modeling nesterov's accelerated gradient
  method: Theory and insights.
\newblock {\em Journal of Machine Learning Research\/}, {\bf 17}, 1--43.

\bibitem[Tautenhahn(1994)Tautenhahn]{Tautenhahn-1994}
{\sc Tautenhahn, U.} (1994)
\newblock On the asymptotical regularization of nonlinear ill-posed problems.
\newblock {\em Inverse Problems\/}, {\bf 10}, 1405--1418.

\bibitem[Tolksdorf(1984)Tolksdorf]{Tolksdorf-1984}
{\sc Tolksdorf, P.} (1984)
\newblock Regularity for a more general class of quasilinear elliptic
  equations.
\newblock {\em J. Differential Equations\/}, {\bf 12}, 126--150.

\bibitem[Vainikko \& Veretennikov(1986)Vainikko \& Veretennikov]{Vainikko1986}
{\sc Vainikko, G. \& Veretennikov, A.} (1986)
\newblock {\em Iteration Procedures in Ill-Posed Problems\/}.
\newblock Nauka (In Russian).

\bibitem[Vazquez(1984)Vazquez]{Vazquez-1984}
{\sc Vazquez, J.} (1984)
\newblock A strong maximum principle for some quasilinear elliptic equations.
\newblock {\em Appl. Math. Optim.}, {\bf 12}, 191--202.

\bibitem[Zhang {\em et~al.}(2018a)Zhang, Gong, Gulliksson, \&
  Cheng]{ZhangJIIP2018}
{\sc Zhang, Y., Gong, R., Gulliksson, M. \& Cheng, X.} (2018a)
\newblock A coupled complex boundary expanding compacts method for inverse
  source problems.
\newblock {\em J. Inverse Ill-Pose. P.}, {\bf DOI}, 10.1515/jiip--2017--0002.

\bibitem[Zhang {\em et~al.}(2018b)Zhang, Gong, Cheng, \&
  Gulliksson]{ZhangYe2018}
{\sc Zhang, Y., Gong, R., Cheng, X. \& Gulliksson, M.} (2018b)
\newblock A dynamical regularization algorithm for solving inverse source
  problems of elliptic partial differential equations.
\newblock {\em Inverse Problems\/}, {\bf 34}, 065001.

\bibitem[Zhang \& Hofmann(2018)Zhang \& Hofmann]{ZhangHof2018}
{\sc Zhang, Y. \& Hofmann, B.} (2018)
\newblock On the second order asymptotical regularization of linear ill-posed
  inverse problems.
\newblock {\em Appl. Anal.}, {\bf DOI}, 10.1080/00036811.2018.1517412.

\end{thebibliography}


\appendix

\section*{Appendix A.\ Proof of Theorem \ref{ThmDynamic}}
\label{app1}

Denote $q^\delta=\dot{p}^\delta$, $q^\delta(0)=\dot{p}^\delta(x,0)$, and rewrite (\ref{DDBNandDDS}) as
\begin{equation}\label{FirstOrder}
\left\{\begin{array}{l}
\dot{p}^\delta = q^\delta, \\
\dot{q}^\delta = -\eta q^\delta - w_{im} \chi_{\Omega_0}, \\
p^\delta(0)=p_{0}, q^\delta(0)=\dot{p}_{0}.
\end{array}\right.
\end{equation}

By inequality (\ref{inequalityw}) in Lemma \ref{LemmaV}, $w_{im} \chi_{\Omega_0}$ is continuously dependent on the source term $p$,
hence, by the Cauchy-Lipschitz theorem, the first order nonautonomous system (\ref{FirstOrder})
has a unique global solution for the given initial data $(p_0, \dot{p}_0)$. Furthermore, by the standard arguments in elliptic PDEs theory~\cite{Cheng:2014,Johnson-2009},
the global existence of the source function $p^\delta(x,t)$ implies the existence and uniqueness
of the elliptic PDEs (\ref{prow}) and (\ref{prou}), which completes the proof of the
global existence and uniqueness of the systems (\ref{DDBNandDDS})-(\ref{prou}).

Now, we show the continuity of the solution $p^\delta$ with respect to the boundary data.

For any fixed $t$, define operator $\mathcal{A}: P\to \textbf{H}^1(\Omega)$ through
$\mathcal{A}p(\cdot,t)=\hat{u}(\cdot,t)$ with $\hat{u}(\cdot,t)\in\textbf{H}^1(\Omega)$ being the weak solution of
\[
\left\{\begin{array}{ll}
-\triangle \hat{u}(x,t) +   \hat{u}(x,t)  = p(x,t) \chi_{\Omega_0}, & x\in\Omega,~ t\in (0,\infty), \\
\frac{\partial \hat{u}(x,t)}{\partial \mathbf{n}} + i \hat{u}(x,t) = 0,
& x\in\Gamma,~ t\in (0,\infty).
\end{array}\right.
\]
Denote by $g=g_2+ i g_1$. For any $g\in \textbf{L}^2(\Gamma)$, define operator $\mathcal{B}: \textbf{L}^2(\Gamma)\to \textbf{H}^1(\Omega)$ through
$\mathcal{B}g=\tilde{u}$, where $\tilde{u}\in\textbf{H}^1(\Omega)$ solves
\[
\left\{\begin{array}{ll}
-\triangle \tilde{u}(x) +   \tilde{u}(x)  = 0, & x\in\Omega, \\
\frac{\partial \tilde{u}(x)}{\partial \mathbf{n}} + i \tilde{u}(x) = g,
& x\in\Gamma.
\end{array}\right.
\]

Furthermore, for any $v\in \textbf{H}^1(\Omega)$, we define $I_m: \textbf{H}^1(\Omega)\to H^1(\Omega)$ through $I_m v=v_{im}$. Following standard arguments in the classical PDEs theory, all of $\mathcal{A}, \mathcal{B}$ and $I_m$ are bounded in the corresponding spaces. One the other hand, if we denote $g^\delta=g_2^\delta+i\,g_1^\delta$, we have
\[
w_{im}=I_m w=I_m\mathcal{A}I_m (\mathcal{A}p^\delta+\mathcal{B}g^\delta) =: \mathcal{M}p^\delta+\mathcal{N}g^\delta.
\]
Substitute the above equation  into (\ref{DDBNandDDS}) to obtain
\begin{eqnarray*}\label{DDBNandDDS1}
\left\{\begin{array}{ll}
\ddot{p}^\delta(x,t) + \eta \dot{p}^\delta(x,t) + \mathcal{M} p^\delta(x,t)=-\mathcal{N}g^\delta, & x\in\Omega_0, t\in (0,\infty), \\
p^\delta(x,0)=p_0, \dot{p}^\delta(x,0)=\dot{p}_0, & x\in\Omega_0.
\end{array}\right.
\end{eqnarray*}

If we define $\delta p = p^\delta-p$, it solves
\begin{eqnarray*}\label{DDBNandDDS2}
\left\{\begin{array}{ll}
\ddot{\delta p}(x,t) + \eta \dot{\delta p}(x,t) + \mathcal{M} \delta p(x,t)=-\mathcal{N}(g^\delta-g), & x\in\Omega_0, t\in (0,\infty), \\
\delta p(x,0)=\dot{\delta p}(x,0)=0, & x\in\Omega_0,
\end{array}\right.
\end{eqnarray*}
Applying the Cauchy-Lipschitz theorem again to deduce that for any fixed $t$, $\delta p(\cdot, t)\to 0$ in $P$ when $g^\delta \to g$ in $\textbf{L}^2(\Gamma)$. Consequently, $p^\delta(\cdot,t)\to p(\cdot,t)$ in $P$ as $\delta\to 0$.

\section*{Appendix B.\ Proof of Proposition \ref{LimitNoisyData}}
\label{app2}

The case with the damping parameter $\eta(t)=r/t$ can be performed along the lines and using the tools of the proof of Lemma \ref{LemmaVanishingExactCase2}. Hence, it suffices to show the case with the fixed damping parameter $\eta(t)=\eta$.

Denote by $p^\delta(t)=p^\delta(x,t)$, and define the Lyapunov function of the differential equation (\ref{DDBNandDDS}) by $\mathcal{E}(t) = V(p^\delta(t))+ \frac{1}{2} \|\dot{p}^\delta(t)\|^2_P$. Similar to the proof of Lemma \ref{LemmaVanishingExact}, we have
\begin{equation}
\label{Decreasing}
\dot{\mathcal{E}}(t) = - \eta \|\dot{p}^\delta(t)\|^2_P.
\end{equation}
Hence, $\mathcal{E}(t)$ is non-increasing, and consequently, $\|\dot{p}^\delta(t)\|^2_P\leq 2\mathcal{E}(0)$. Therefore, $\dot{p}^\delta(\cdot)$ is uniform bounded.
Integrating both sides in (\ref{Decreasing}), we obtain
\begin{eqnarray*}
\int^{\infty}_{0} \|\dot{p}^\delta(t)\|^2_P dt \leq  \mathcal{E}(0) / \eta <  \infty,
\end{eqnarray*}
which yields $\dot{p}^\delta(\cdot) \in L^2([0,\infty),P)$.

Now, let us show that for any $p^\dagger\in P$ the following inequality holds.
\begin{eqnarray}\label{supLimit}
\mathop{\lim\sup}_{t\to \infty} V(p^\delta(t)) \leq V(p^\dagger).
\end{eqnarray}

Consider for every $t\in[0,\infty)$ the function $e(t)=e(t;p^\dagger):=\frac{1}{2} \|p^\delta(t) - p^\dagger\|^2_P$. Since $\dot{e}(t)= ( p^\delta(t) - p^\dagger, \dot{p}^\delta(t) )_P $ and $\ddot{e}(t)= \|\dot{p}^\delta(t)\|^2_P+ ( p^\delta(t) - p^\dagger, \ddot{p}^\delta(t) )_P$ for every $t\in[0,\infty)$. Taking into account (\ref{DDBNandDDS}), we get
\begin{equation}
\label{EqError2Sec3}
\ddot{e}(t) + \eta \dot{e}(t) + ( p^\delta(t) - p^\dagger,  u_{im}(p^\delta(t)) )_P = \|\dot{p}^\delta(t)\|^2_P.
\end{equation}

On the other hand, by the convexity inequality of the residual norm square functional $V(p^\delta(t))$, we derive
\begin{eqnarray}\label{convexityIneqSec3}
V(p^\delta(t))  + ( p^\dagger - p^\delta(t), \nabla V(p^\delta(t)) )_P \leq V(p^\dagger).
\end{eqnarray}

Combine (\ref{EqError2Sec3}) and (\ref{convexityIneqSec3}) with the definition of $\mathcal{E}(t)$ to obtain
\begin{eqnarray*}
\ddot{e}(t) + \eta \dot{e}(t) \leq V(p^\dagger) - \mathcal{E}(t) + \frac{3}{2} \|\dot{p}^\delta(t)\|^2_P.
\end{eqnarray*}
By (\ref{Decreasing}), $\mathcal{E}(t)$ is non-increasing, hence, given $t>0$, for all $\tau\in[0,t]$ we have
\begin{eqnarray*}
\label{ProofIneq}
\ddot{e}(\tau) + \eta \dot{e}(\tau) \leq V(p^\dagger) - \mathcal{E}(t) + \frac{3}{2} \|\dot{p}^\delta(\tau)\|^2_P.
\end{eqnarray*}
By multiplying this inequality with $e^{\eta \tau}$ and then integrating from 0 to $\theta$, we obtain
\begin{eqnarray*}
\dot{e}(\theta)\leq e^{-\eta \theta} \dot{e}(0) + \frac{1-e^{-\eta \theta}}{\eta} (V(p^\dagger) - \mathcal{E}(t)) + \frac{3}{2} \int^\theta_0 e^{-\eta(\theta-\tau)} \|\dot{p}^\delta(\tau)\|^2_P d \tau .
\end{eqnarray*}
Integrate the above inequality once more from 0 to $t$ together with the fact that $\mathcal{E}(t)$ decreases, to obtain
\begin{eqnarray}\label{ProofIneq2}
e(t)\leq e(0)+ \frac{1-e^{-\eta t}}{\eta} \dot{e}(0) + \frac{\eta t - 1 + e^{-\eta t}}{\eta^2} (V(p^\dagger) - \mathcal{E}(t)) + h(t),
\end{eqnarray}
where $h(t):= \frac{3}{2} \int^t_0 \int^\theta_0 e^{-\eta(\theta-\tau)} \|\dot{p}^\delta(\tau)\|^2_Pd \tau d \theta$.

Since $e(t)\geq0$ and $\mathcal{E}(t)\geq V(p^\delta(t))$, it follows from (\ref{ProofIneq2}) that
\begin{eqnarray*}
\frac{\eta t - 1 + e^{-\eta t}}{\eta^2} V(p^\delta(t)) \leq e(0)+ \frac{1-e^{-\eta t}}{\eta} \dot{e}(0) + \frac{\eta t - 1 + e^{-\eta t}}{\eta^2} V(p^\dagger) + h(t).
\end{eqnarray*}
Dividing the above inequality by $\frac{\eta t - 1 + e^{-\eta t}}{\eta^2}$ and letting $t\to\infty$, we deduce that
\begin{eqnarray*}
\mathop{\lim\sup}_{t\to \infty} V(p^\delta(t)) \leq V(p^\dagger) + \mathop{\lim\sup}_{t\to \infty} \frac{\eta}{t} h(t).
\end{eqnarray*}

Hence, for proving (\ref{supLimit}), it suffices to show that $h(\cdot)\in L^\infty([0,\infty),\mathcal{X})$. It is obviously held by noting the following inequalities
\begin{eqnarray*}
0 \leq h(t) = \frac{3}{2\eta} \int^t_0 (1- e^{-\eta(t-\tau)}) \|\dot{p}^\delta(\tau)\|^2_Pd \tau \leq \frac{3}{2\eta} \int^\infty_0 \|\dot{p}^\delta(\tau)\|^2_Pd \tau < \infty.
\end{eqnarray*}
From the inequality $V(p^\delta(t)) \geq \inf_{p^\dagger\in P} V(p^\dagger)$, we conclude together with (\ref{supLimit}) that
\begin{equation}
\label{ProofLimit1}
\lim_{t\to \infty} V(p^\delta(t))  =  \inf_{p^\dagger\in P} V(p^\dagger).
\end{equation}
Consequently, we have
\begin{eqnarray*}
\lim_{t\to \infty} V(p^\delta(t)) \leq V(p^\dagger) \leq C^2_0 \delta^2.
\end{eqnarray*}

\end{document}